\documentclass[10pt]{amsart}
\usepackage{amsmath,amssymb,verbatim}
\usepackage[utf8]{inputenc}
\usepackage{enumerate}
\usepackage{pdflscape}
\usepackage{amsthm}
\usepackage{mathrsfs}
\usepackage{color}
\usepackage{multicol}
\usepackage[normalem]{ulem}
\usepackage{cancel}
\usepackage{tikz}
\usetikzlibrary{matrix}
\usepackage[all]{xy}
\usepackage{mathtools}
\usepackage{arydshln}
\usepackage[shortlabels]{enumitem}
 \usepackage[english]{babel}
\usepackage{hyperref}
\hypersetup{
    colorlinks = true,
    linkbordercolor = {white},
    allcolors = blue
} 
\usepackage[capitalize, noabbrev]{cleveref}
 
\makeatletter
\@namedef{subjclassname@1991}{\textup{2020} Mathematics Subject Classification}
\makeatother

\topskip=\baselineskip 
\newtheorem{theorem}{Theorem}[section]
\newtheorem{lemma}[theorem]{Lemma}

\newtheorem{proposition}[theorem]{Proposition}

\newtheorem{corollary}[theorem]{Corollary}
\newtheorem{remark}[theorem]{Remark}

\newtheorem*{theorem*}{Theorem}

\newtheorem{maintheorem}{Theorem}
\newtheorem{maincorollary}[maintheorem]{Corollary}

\makeatletter
\@addtoreset{equation}{section}
\makeatother

\newcommand{\M}{{\mathrm{D}}}

\newcommand{\bb}{{\overline b}}

\newcommand{\Cen}{{\rm C}}

\newcommand{\Aut}{\mbox{\rm Aut}}

\newcommand{\inv}{\textup{inv}}

\newcommand{\Z}{{\mathbb Z}}

\newcommand{\F}{{\mathbb F}}

\newcommand{\G}{\mathcal{G}}
\newcommand{\B}{\mathcal{B}}

\newcommand{\Soc}{{\rm Soc}}

\newcommand{\GEN}[1]{\left\langle #1 \right\rangle}

\newcommand{\ov}[1]{\overline{#1}}
\newcommand{\aug}[1]{\mathrm{I}(#1)}
\newcommand{\augNor}[2]{\mathrm{I}(#1;#2)}

\newcommand{\ZZ}{\mathrm{Z}}

\newcommand{\E}[1]{\left\lceil #1 \right\rceil}

\newcommand{\ord}{\text{o}}
\newcommand{\lex}{\text{lex}}

\newcommand{\Ese}[2]{\mathcal{S}\left(#1\mid #2\right)}
\newcommand{\Te}[2]{\mathcal{T}\left(#1\mid #2\right)}

\newcommand{\qand}{\quad \text{and} \quad}

\DeclareMathOperator{\Cyc}{C}

\DeclareMathOperator{\OO}{O}
\DeclareMathOperator{\dG}{d}

\title[On group invariants determined by modular group algebras]{On group invariants determined by modular group algebras: even versus odd characteristic}

\author{Diego Garc\'{\i}a-Lucas, \'{A}ngel del R\'{i}o, Mima Stanojkovski}
\thanks{The first two authors   are  partially supported by Grant PID2020-113206GB-I00 funded by MCIN/AEI/10.13039/501100011033.
The third author is supported by the Deutsche Forschungsgemeinschaft (DFG, German Research Foundation) -- Project-ID 286237555 -- TRR 195.}



\date{\today}

\begin{document}

\begin{abstract}
Let $p$ be a an odd prime and let $G$ be a finite $p$-group with cyclic commutator subgroup $G'$. We prove that 
the exponent and the abelianization of the centralizer of $G'$ in $G$ are determined by the group algebra of $G$ over any field of characteristic $p$. If, additionally, $G$ is $2$-generated then almost all the numerical invariants determining $G$ up to isomorphism are determined by the same group algebras; as a consequence the isomorphism type of the centralizer of $G'$ is determined. These claims are known to be false for $p=2$. 
\end{abstract}

\maketitle 


\section{Introduction}

Let $G$ be a group, $R$ a commutative ring and let $RG$ denote the group  ring of $G$ with coefficients in $R$. The problem of describing how much information about the group $G$ is carried by the group algebra $RG$ has a long tradition in mathematics, with applications in particular to the representation of groups and in general to group theory; cf.\ \cite{Hig40,PW50,Bra63,Passman65,
Sehgal1967,Whitcomb,Pas77,Seh78,CurtisReiner1981,
Sandling85,CurtisReiner1987,BKRW99}.
The last question can be rewritten more compactly as: 
\begin{quote}
Which group invariants of $G$ are algebra invariants of $RG$?
\end{quote}
By a \emph{group invariant} of $G$ we understand a feature of $G$ that
is shared with any group isomorphic to $G$ while an \emph{algebra invariant} is a feature that is shared with any group $H$ with the property that $RG$ and $RH$ are isomorphic as $R$-algebras.
For instance, the cardinality of $G$ can be expressed as the $R$-rank of $RG$ and is thus an algebra invariant of $RG$. Moreover, the group $G$ is abelian if and only if $RG$ is a commutative ring, i.e.\ the property of being abelian is an algebra invariant of $RG$. 
The ultimate version of the above question is the \emph{Isomorphism Problem} which asks for the determination of pairs $(G,R)$ for which the isomorphism type of $G$ is an algebra invariant of $RG$:
\begin{quote}
	\textbf{Isomorphism Problem for group algebras}: 
	Given a commutative ring $R$ and two groups $G$ and $H$,  does  $RG$ and $RH$ being isomorphic as $R$-algebras imply that the groups $G$ and $H$ are isomorphic? In symbols,
	\[
 RG \cong RH\ \Longrightarrow\ G\cong H\ ?	
	\]
\end{quote}
The answer to this question is a function of the ring $R$: for instance, it is easily shown that any two non-isomorphic finite abelian groups of the same order have isomorphic group rings with complex coefficients. However, by a seminal result of G. Higman \cite{Hig40}, if $G$ and $H$ are non-isomorphic abelian groups then $\Z G$ and $\Z H$ are not isomorphic. More surprisingly, there even exist two non-isomorphic finite metabelian groups $G$ and $H$ such that $kG$ and $kH$ are isomorphic for every field $k$ \cite{Dade71}. Nonetheless, the Isomorphism Problem has a positive solution for $R=\Z$ and $G$ and $H$ metabelian \cite{Whitcomb}. This extends Higman's result for abelian groups \cite{Hig40} and has been followed by positive results for more families of groups, such as nilpotent groups \cite{RoggenkampScott1987} and supersolvable groups \cite{KimmerleHabil}. These early results yielded to strong expectations that the Isomorphism Problem for integral group rings ($R=\Z$ and $G$, $H$ finite) would have a positive solution until Hertweck's construction of two non-isomorphic finite groups with isomorphic integral group rings \cite{Hertweck2001}. Among the classical variations of the Isomorphism Problem, the one that remained unanswered the longest deals with the case where  $R$  is a field of positive characteristic $p$ and $G$ and $H$ are $p$-groups, formally:
\begin{quote}
	\textbf{Modular Isomorphism Problem}: Given a field $k$ of characteristic $p>0$ and two finite $p$-groups $G$ and $H$, are $kG$ and $kH$ isomorphic as $k$-algebras if and only if $G$ and $H$ are isomorphic as groups?
\end{quote}
 The contributions to this problem are numerous, including positive solutions for specific families of $p$-groups and the uncovering of algebra invariants in this context; cf.\ 
\cite{Jen41,Deskins1956,Ward,Qui68,Makasikis,BaginskiMetacyclic,BC88,
San84Ab,San89,Wursthorn1993,SalimSandling1995,SalimSandlingp5,San96,Bag99,
HS06,BK07,HertweckSoriano07,Eick08,EK11,NavarroSambale,
Sakurai,MM20,BdR20,MSS21,MS22}. The first negative solution to the Modular Isomorphism Problem was given recently in the form of a series of pairs of non isomorphic $2$-groups $G_{m,n}$ and $H_{m,n}$  which are $2$-generated and have cyclic commutator subgroup satisfying $kG_{m,n}\cong kH_{m,n}$ for every $n>m>2$ and every field $k$ of characteristic $2$ \cite{GarciaMargolisdelRio}. However, if $p$ is odd then the Modular Isomorphism Problem is still open, even in the class of $2$-generated groups with cyclic commutator subgroup. 
The aim of this paper is to investigate this class of groups from the point of view of algebra invariants and to 
demonstrate a substantial difference between the cases $p=2$ and $p>2$ within this class.
For example, if $G'$ denotes the commutator subgroup of $G$, our first result shows that both the exponent of $\Cen_G(G')$ and the isomorphism types of $\Cen_G(G')/G'$ and $\Cen_G(G')/\Cen_G(G')'$ are  algebra invariants of $kG$ provided $G'$ is cyclic and $p$ is odd; cf.\ \cref{ExponentCGG'}.
Note that, on the contrary, for every choice of $n>m>2$,
the groups $G_{m,n}$ and $H_{m,n}$ satisfy neither of the points $(1)$-$(2)$-$(3)$ from \cref{ExponentCGG'}. 

\begin{maintheorem}\label{ExponentCGG'}
Let $k$ be a field of odd characteristic $p$ and let $G$ and $H$ be finite $p$-groups. If $G'$ is cyclic and $kG\cong kH$ then the following hold: 
\begin{enumerate}[label=$(\arabic*)$]
		\item $\Cen_G(G')$ and $\Cen_H(H')$ have the same exponent. 
		\item $\Cen_G(G')/G'\cong \Cen_H(H')/H'$.
		\item  $\Cen_G(G')/\Cen_G(G')'\cong \Cen_H(H')/\Cen_H(H')'$.   
\end{enumerate}
	\end{maintheorem}
Our next results concern 2-generated $p$-groups with cyclic commutator subgroup. In order to present them we introduce some numerical invariants of these groups. Since the Modular Isomorphism Problem has a positive solution for abelian groups \cite{Deskins1956} we only consider non-abelian groups. 
To this end, let $G$ be a $2$-generated non-abelian $p$-group, that is, $G'$ is non-trivial and $G$ is generated by exactly $2$ elements. 
A \emph{basis} of $G$ is then a pair $(b_1,b_2)$ of elements of $G$ such that 
\[
G/G'=\GEN{b_1G'}\times \GEN{b_2G'} \textup{ and } |b_2G'| \textup{ divides } |b_1G'|.\]
Moreover, we define 
$$\OO(G)=\min\mbox{}_{\lex} \{ (|b_1\Cen_G(G')|, |b_2\Cen_G(G')|, -|b_1|, -|b_2|)  : (b_1,b_2) \textup{ is a basis of } G \},$$
where $\min_{\lex}$ refers to the minimum with respect to the lexicographic order. 
The following result implies that if $G'$ is cyclic then so is $H'$ and $\OO(G)$ is an algebra invariant of $kG$. 

\begin{maintheorem}\label{OTheorem}
	Let $k$ be a field of odd characteristic $p$ and let $G$ and $H$ be finite $p$-groups with $kG\cong kH$. If $G$ is $2$-generated and $G'$ is cyclic, then $H$ is $2$-generated, $H'$ is cyclic and $\OO(G)=\OO(H)$. 
\end{maintheorem}
Observe that the hypothesis that $p$ is odd in \cref{OTheorem} is necessary because 
$$\OO(G_{m,n})=(2,2,-2^n,-2^m) \ne (2,1,-2^n,-2^m)=\OO(H_{m,n}).$$ 
This shows again a clear contrast between odd and even characteristic. 

For any $p$-group $G$ that is $2$-generated and for which $G'$ is cyclic, the vector $\OO(G)$ can essentially be extended to a vector 
$\inv(G)=(p,m,n_1,n_2,\sigma_1,\sigma_2,o_1,o_2,o'_1,o'_2,u_1,u_2)$ of numerical invariants characterizing the isomorphism class of $G$; cf.\ \Cref{SectionClassification} and \cite{OsnelDiegoAngel}. 
It is well known that the first four entries $p,m,n_1$ and $n_2$ of $\inv(G)$ are algebra invariants of $kG$. However, the sixth entry is not determined by the modular group algebra because $\inv(G_{m,n})=(2,2,n,m,-1,-1,0,0,0,0,1,1)$ is different from $\inv(H_{m,n})=(2,2,n,m,-1,1,0,0,0,0,1,1)$. 
Note that, for $p>2$, one always has $\sigma_1=\sigma_2=1$ and therefore the counterexample from \cite{GarciaMargolisdelRio} does not have a direct equivalent in odd characteristic. 
We will see that \Cref{OTheorem} is actually equivalent to the following. 

\begin{maintheorem}\label{oo'Theorem}
	Let $k$ be a field of odd characteristic $p$ and let $G$ be a  finite non-abelian $p$-group.
	If $G$ is $2$-generated with $G'$ cyclic then all but the last $2$ entries of $\inv(G)$ are algebra invariants of $kG$.
\end{maintheorem}

In other words, \cref{oo'Theorem} ensures that, for $p>2$, the first $10$ entries of $\inv(G)$ are determined by the modular group algebra $kG$ of $G$ over any field $k$ of characteristic $p$.
Unfortunately we have not been able to decide whether the last two entries of $\inv(G)$ are algebra invariants of $kG$. 
The smallest groups for which \Cref{oo'Theorem} does not solve the Modular Isomorphism Problem occur for
$(n_1,n_2,m,o_1,o_2,o_1',o_2')=(3,2,2,0,1,1,1)$,
in which case $u_2=1$ and $u_1\in \{1,\dots,p-1\}$. 
That is, the last parameters yield $p-1$ non-isomorphic 2-generated $p$-groups with cyclic commutator subgroup.
In a paper in preparation, we develop new techniques (different from those presented in this paper) to prove that, for this special case and many others, $u_1$ is actually also an algebra invariant of $kG$. 
\Cref{oo'Theorem} enables us, however, to improve \Cref{ExponentCGG'} in the $2$-generated case by showing that  the isomorphism type of $\Cyc_G(G')$ is an algebra invariant of $kG$: 

\begin{maincorollary}\label{Cor1.0}
	Let $k$ be a field of odd characteristic $p$  and let $G$ be a finite $2$-generated $p$-group  with $G'$ cyclic. Then the isomorphism type of  $\Cen_G(G') $ is an algebra invariant of $kG$.
\end{maincorollary}

The following corollary also follows   from \Cref{oo'Theorem} 
(see \Cref{SectionPreliminaries} for the definition of the type invariants of a $p$-group). 

\begin{maincorollary}\label{Cor1.1}
Let $k$ be a field of odd characteristic $p$ and let $G$ be a finite $2$-generated $p$-group with $G'$ cyclic. Then the type invariants of $G$ are algebra invariants of $kG$.
\end{maincorollary}

The paper is organized as follows. In \Cref{SectionPreliminaries} we establish the notation, recall some known facts about the Modular Isomorphism Problem and prove a key lemma which we refer to as the Transfer Lemma (\Cref{mainlemma}). In \Cref{SectionClassification} we recall the classification of finite 2-generated $p$-groups with cyclic commutator subgroup from \cite{OsnelDiegoAngel} in the specific case where $p>2$. Additionally, we prove a series of results about these groups which will be used in the next and final section to prove the main results of the paper.

\vspace{10pt}

\section{Notation and preliminaries}\label{SectionPreliminaries}

In this section, we introduce the notation that will be used throughout this paper. We also collect some classical results on the Modular Isomorphism Problem that will be useful in the coming sections, as well as a new criterion for the transfer of ideals between modular group algebras; cf.\ \cref{mainlemma}.

Throughout the paper, $p$ will denote a prime number, $k$ a field of characteristic $p$ and $G$ and $H$ finite $p$-groups. The modular group algebra of $G$ over $k$ is denoted by $kG$ and the  augmentation ideal of $kG$ is denoted by $\aug{G}$. It is a classical result that $\aug{G}$ is also the Jacobson ideal of $kG$. For every normal subgroup $ N$ of $ G$, we write $\augNor{N}{G}$ for the relative augmentation ideal $\aug{N}kG$. 

We let $\le_{\lex}$ denote the lexicographic order on tuples of integers of the same length. 
Then $\min_{\lex}$ and $\max_{\lex}$ stand for minimum and maximum with respect to $\le_{\lex}$.
For a non-zero integer $n$, let $v_p(n)$ denote the $p$-adic  valuation of $n$, that is, the greatest integer $t$ such that $p^t$ divides $n$. 
Moreover, set  $v_p(0)=+\infty$. 
For coprime integers $m$ and $n$, write  $o_m(n)$ for  the multiplicative order of $n$ modulo $m$, i.e.\ the smallest  non-negative  integer $k$ with $n^k\equiv 1 \bmod m$.
Given  non-zero  integers $s,t$ and $n$ with $n \ge 0$ we set
$$\Ese{s}{n} = \sum_{i=0}^{n-1} s^i \qand \Te{s,t}{n} = \sum_{0\le  	i < j < n} s^i t^j.$$
The last notation allows us, in some cases, to compactly express powering of products in a group $G$. For instance, if $g,h\in G$ and $r,s,n$ are integers with $n\geq 0$ then, writing $a=[h,g]=h^{-1}g^{-1}hg$, we get the following identities: 
\begin{equation}\label{eq:conj-r}
\textup{if } g^h=h^{-1}gh=g^r \textup{ then } (hg)^n=h^ng^{\Ese{r}{n}},
\end{equation}
\begin{equation}\label{eq:double-conj}
\textup{if } a^g=a^r \textup{ and } a^h=a^s \textup{ then } (gh)^n=g^nh^na^{\Te{r,s}{n}}.
\end{equation}
The next lemma describes elementary properties of the operators  $\mathcal{S}$ and $\mathcal T$ that are collected in Lemmas 8.2 and 8.3 of \cite{OsnelDiegoAngel}. 

\begin{lemma}\label{EseProp}
	Let $p$ be an odd prime number and let $n>0$ and $s,t$ be integers satisfying $s\equiv t \equiv 1 \bmod p$. Then the following hold:
	\begin{enumerate}[label=$(\arabic*)$]
		\item\label{ValEse} $v_p(s^n-1)=v_p(s-1)+v_p(n)$, $v_p(\Ese{s}{n})=v_p(n)$ and $\ord_{p^n}(s)={p^{\max(0,n-v_p(s-1))}}$.   
		\item\label{EseCero} 
		if $v_p(s-1)=a$ and $p^{m-a}$ divides $n$ 
		then $ \Ese{s}{n} \equiv n \bmod p^m $. 
		\item\label{ValEse2}
		$\Te{s,t}{p^n} \equiv 0 \bmod p^n $.
	\end{enumerate}
\end{lemma}

\begin{lemma}\label{EseBijective}
		Let $p$ be an odd prime number and $m$ and $r$ be integers with $m>0$ and $r\equiv 1 \mod p$. Then for every integer $0\le x < p^m$ there is a unique integer $0\le y < p^m$ such that $\Ese{r}{y}\equiv x \mod p^m$. 
	\end{lemma}
	\begin{proof}
		Let $x$ and $y$ be integers with $0\le x \le y$. Then $\Ese{r}{y}-\Ese{r}{x}=r^x\Ese{r}{y-x}$ and hence  from  \Cref{EseProp}\ref{ValEse} it follows that $\Ese{r}{x}\equiv \Ese{r}{y} \mod p^m$ if and only if $x\equiv y \mod p^m$. This shows that $\Ese{r}{\cdot}$ induces an injective, thus bijective, map $\Z/{p^m}\Z\to \Z/{p^m}\Z$.
		 Then the result follows immediately.
	\end{proof}

The group theoretic notation we use is mostly standard. For an arbitrary group $G$, let $|G|$ denote its order, $\ZZ(G)$ its center, $\{\gamma_i(G)\}_{i\geq 1}$ its lower central series, $G'=\gamma_2(G)$,  its commutator subgroup, $\exp(G)$ its exponent and  $\dG(G)=\min \{|X|   :  X\subseteq G \textup{ and } G=\GEN{X}\}$, its minimum number of generators. Moreover, if $g\in G$ and $X\subseteq G$ then $|g|$ denotes the order of $g$ and $\Cen_G(X)$ the centralizer of $X$ in $G$.
We write $\times$ both for internal and external direct products of groups.
For $n\geq 1$, we denote by $C_n$   the cyclic group of order $n$.  

Let  now  $G$ be a finite $p$-group and let $p^e=\exp(G)$. For every $0\le n\le e$ we   define   the following subgroups of $G$:
$$\Omega_n(G)=\GEN{g\in G : g^{p^n}=1} \qand 
\mho_n(G)=\GEN{g^{p^n} : g \in G}.$$   
If $N$ is a normal subgroup of $G$,   we also write
$$\Omega_n(G:N)=\GEN{g\in G: g^{p^{n}}\in N}, $$ that is, $\Omega_n(G:N)$ is the only subgroup of $G$ containing $N$ such that 
$$\Omega_n(G:N)/N=\Omega_n(G/N). $$ 
 The group  $G$ is said to be \emph{regular} if for every $g,h\in G$  there exist $c_1,\dots,c_k\in \GEN{g,h}'$ such that  $(gh)^p=g^ph^pc_1^p\cdots c_k^p$, in other words  
 	$$(gh)^p\equiv g^ph^p\bmod \mho_1(\GEN{g,h}').$$ 
It is well known that if $p$ is odd and $G'$ is cyclic then $G$ is regular, while this is not the case for $p=2$; cf. \cite[Satz~III.10.2, Satz~III.10.3(a)]{Hup67}. Moreover, if $G$ is regular, \cite[Hauptsatz~III.10.5, Satz~III.10.7]{Hup67} ensure that the following hold:
\begin{itemize}
\item $\Omega_n(G) = \{g\in G : g^{p^n}=1\}$  and $\mho_n(G)= \{g^{p^n} : g \in G\}$,
\item $|\Omega_n(G)|\cdot|\mho_n(G)|=|G|$ for every $0\le n \le e$.
\end{itemize} 
For every $n\ge 1$ and $G$ regular we define $w_n$ by means of
	$$p^{\omega_n}=|\Omega_n(G)/\Omega_{n-1}(G)|=|\mho_{n-1}(G)/\mho_n(G)|$$
	and remark that $\omega_1\ge \omega_2\ge \ldots \ge \omega_e>0$. Following \cite[\S~III.10]{Hup67} we set $\omega(G)=\omega_1$  and, for $1\le i \le \omega(G)$, we define
$$e_i=|\{1\le n \le e : \omega_n \ge i\}|.$$
 It follows that $e=e_1\ge e_2\ge \ldots \ge e_{\omega(G)}$ and the entries of the list $(e_1,\dots,e_{\omega(G)})$ are called the \emph{type invariants} of $G$. 
 
  The \emph{Jennings series} $(\M_n(G))_{n\geq 1}$ of the $p$-group $G$ is defined by 
\begin{equation}
\label{EqJenningsLazard}
\M_{n}(G) = \{g\in G : g-1\in \aug{G}^i\}= \prod_{ip^j\ge n}   \mho_j(\gamma_i(G)) 
\end{equation} 
The Jennings series is also known as the \emph{Brauer-Jennings-Zassenhaus series} or the \emph{Lazard series}
or the \emph{dimension subgroup series} of $kG$
 (see \cite[Section~11.1]{Pas77} for details).   A property of these series that we will use is that, for abelian groups, the orders of the terms completely determine the structure of the group.  
 For more on the Jennings series, see for instance \cite[Section~III.1]{Seh78}, \cite[Section~4]{Mar22} and \cite[Section~2.3]{MS22}.

The next proposition and lemma collect some well-known results which will be used throughout the paper.

\begin{proposition}\label{Known}
Let $k$ be a field of positive characteristic $p$ and let $G$ be a finite $p$-group.
Then the following statements hold: 
\begin{enumerate}[label=$(\arabic*)$]
\item\label{ZGG'Can} If $H$ is a finite $p$-group and $\phi:kG\rightarrow kH$ is an isomorphism of $k$-algebras 
then  $$\phi(\augNor{G'}{G})=\augNor{H'}{H} \ \textup{ and } \ \phi(\augNor{\ZZ(G)G'}{G})=\augNor{\ZZ(H)H'}{H}.$$
\item\label{it:2K} The following group invariants of $G$ are algebra invariants of $kG$: 
    \begin{enumerate}[label=$(\alph*)$]
    \item  \label{AbelDeter}
    The isomorphism type of $G/G'$.
    \item\label{expDet} The exponent of $G$.    
    \item
    \label{JenningsDet} 
    The isomorphism type of the consecutive quotients $\M_i(G)/\M_{i+1}(G)$ and $\M_i(G')/\M_{i+1}(G')$ of the Jennings series of $G$ and $G'$.
    \item\label{Generators}The minimum number of generators $\dG(G)$ of $G$   and  $\dG(G')$ of $G'$.
    \item\label{Z/ZG'} The isomorphism types of $\ZZ(G)\cap G'$ and $\ZZ(G)/\ZZ(G)\cap G'$.  
    \end{enumerate}
\item\label{it:3K} The Modular Isomorphism Problem has a positive solution in the following cases:
    \begin{enumerate}[label=$(\alph*)$]
    \item\label{it:MIPAbelian} $G$ is abelian.
    \item\label{Metacyclic} $G$ is metacyclic. 
    \item\label{2GenClass2} $G$ is $2$-generated of class $2$ and $p$ is odd.
    \end{enumerate}
\end{enumerate}
\end{proposition}

 \begin{lemma}\label{JenningsNor}
	Let $k$ be a field of characteristic $p>0$, let $G$ and $H$ be  finite $p$-groups and let $L_G$  and $L_H$ be normal subgroups of $G$ and $H$,  respectively. Assume that there is an isomorphism $\phi:kG\to kH$ such that  $ \phi\left(\augNor{L_G}{G} \right)=\augNor{L_H}{H}$. Then,  for each $i\geq 1$, there is an isomorphism of groups\begin{equation*}
	\M_{i}(L_G)/\M_{i+1}(L_G) \cong \M_{i}(L_H)/\M_{i+1}(L_H).
	\end{equation*} 
\end{lemma}

\begin{remark}{\rm 
		Although all the statements of \Cref{Known}   are well known, some of them appear in the literature with the assumption that $k =\F_p $, the field  with $p$ elements, and the proof of others is hidden inside proofs of other statements. We add a few words so that the reader can track the results in the literature.
		
		The proof of \ref{ZGG'Can} can be found inside the proof of \cite[Theorem~2.(ii)]{BK07}.  	Statements 
		\ref{AbelDeter} and \ref{expDet} from \ref{it:2K} are proven in \cite{Sehgal1967,Ward}, \cite{Sandling85, Ward} and 
		\cite{Kulshammer1982}, respectively.
		The statement in \ref{it:2K}\ref{JenningsDet} for the Jenning series of $G$ is proved in \cite[Lemma~14.2.7(i)]{Pas77} while the statement for $G'$ is proved for $k=\F_p$  in \cite[Lemma~2]{BaginskiMetacyclic} and it can be generalized to arbitrary $k$ using the argument from \cite[Lemma~14.2.7(i)]{Pas77}. 
		Observe that the two statements in \ref{it:2K}\ref{JenningsDet} also follow  from  \Cref{JenningsNor} specialized to $L_G=G$ and $L_G=G'$, respectively. 
		Statement \ref{it:2K}\ref{Generators} is a consequence of \ref{it:2K}\ref{JenningsDet} because $[\M_1(G):\M_2(G)]=p^{\dG(G)}$.
		In \cite[Theorem~6.11]{Sandling85} point \ref{it:2K}\ref{Z/ZG'} is stated for $k=\F_p$, but its proof can be easily generalized to hold for any $k$. 
		
		Statement \ref{it:3K}\ref{it:MIPAbelian} is proven in \cite[Theorem~2]{Deskins1956} while statements \ref{Metacyclic} and \ref{2GenClass2} of \ref{it:3K} are proven for $k=\F_p$ in  \cite{BaginskiMetacyclic, San96} and  \cite{BdR20}, respectively. The latter proofs generalize gracefully for any $k$. Note that  the analogue of   \ref{it:3K}\ref{2GenClass2} for $p=2$ appears in \cite{BdR20} for $k =\F_2 $, but the proof does not  generalize to arbitrary $k$.
		
Finally, \Cref{JenningsNor} is proven for $k=\F_p$ in \cite[Lemma~6.26]{Sandling85} and the proof can be easily generalized to hold for any field of characteristic $p$.
}\end{remark}

 We close this section with a lemma that we will make use of to obtain new group algebra invariants from old ones. A version of this lemma, specialized for $N_\Gamma= \Gamma'$ and with a different proof, appears in \cite{SalimTesis}. 

\begin{lemma}[Transfer Lemma]\label{mainlemma}
	Let $p$ be a prime number and  $G$ and $H$ finite $p$-groups.
For $\Gamma\in \{G,H\}$ let $N_\Gamma$ be a subgroup of $\Gamma$ containing $\Gamma'$.  
	If $k$ is a field of characteristic $p$ and $\phi:kG\rightarrow kH$ is a ring isomorphism such that $\phi(\augNor{N_G}{G})=\augNor{N_H}{H}$ then 
	$\phi(\augNor{G}{\Omega_t(G:N_G)})=\augNor{H}{\Omega_t(H:N_H)}$ for every positive integer $t$. 
\end{lemma}

\begin{proof}
Fix $\Gamma\in\{G,H\}$. 
As $\Gamma/N_\Gamma$ is abelian we have $\Omega_t(\Gamma:N_\Gamma)=\{g\in \Gamma : g^{p^t}\in N_\Gamma\}$. Then, if $t\ge 2$ we have  $\Omega_t(\Gamma:N_\Gamma)=\Omega_1(\Gamma:\Omega_{t-1}(\Gamma:N_\Gamma))$ and hence we assume without loss of generality that $t=1$.
For brevity write $L_\Gamma=\mho_1(\Gamma,N_\Gamma)$ and 
we  will  prove that $\phi(\augNor{G}{L_G})=\augNor{H}{L_H}$.
	
Let $\tau_\Gamma:k\Gamma\rightarrow k\Gamma$ be the  ring  homomorphism extending the identity on $\Gamma$ and the Frobenius map $x\rightarrow x^p$ on $k$. 
Moreover, for a normal subgroup $K$ of $\Gamma$, let $\pi_K:k\Gamma \to k(\Gamma/K)$ denote the natural projection with 
$\ker \pi_K=\augNor{K}{\Gamma}$.  
As $\Gamma/N_{\Gamma}$ is abelian, the assignment $g\mapsto g^p$ on $\Gamma$ induces a ring homomorphism  $\lambda_{\Gamma}:k(\Gamma/N_{\Gamma})\rightarrow k(\mho_1(\Gamma/N_{\Gamma}))$. We denote by $\sigma_\Gamma$ the $k$-linear map extending the restriction of $\lambda_{\Gamma}$ to $\Gamma/N_{\Gamma}$. 
By the definition of $L_\Gamma$ and the hypothesis $\phi(\augNor{N_G}{G})=\augNor{N_H}{H}$, we have that $\sigma_\Gamma$ induces an isomorphism $\hat\sigma_\Gamma:k(\Gamma/L_\Gamma)\rightarrow k(\mho_1(\Gamma/N_\Gamma))$ and 
$\phi$ induces an isomorphism $\hat\phi:k(G/N_G)\rightarrow k(H/N_H)$ making the following diagram commute:
$$\xymatrix{
	& k(G/L_G) \ar[r]^{\hat\sigma_G \ \  \ } & k(\mho_1(G/N_G)) \ar[d]^{\tau_G} \\
	kG \ar[d]^{\phi} \ar[r]^{\pi_{N_G} \ \ \ } \ar[ru]^{\pi_{L_G}} & k(G/N_G) \ar[d]^{\hat \phi}  \ar[r]^{\lambda_G \ \ } \ar[ru]^{\sigma_G} &    k(\mho_1(G/N_G))  \ar[d]^{\hat \phi} \\
	kH \ar[r]^{\pi_{N_H} \ \ \ } \ar[rd]^{\pi_{L_H}} & k(H/N_H) \ar[r]^{\lambda_H \ \ } \ar[rd]^{\sigma_H} &  k(\mho_1(H/N_H))  \\
	& k(H/L_H) \ar[r]^{\hat\sigma_H \ \ \ } & k(\mho_1(H/N_H)) \ar[u]_{\tau_H} \\
}$$
As $\phi$ is a bijection, then so is $\hat \phi$. Moreover, each $\hat\sigma_{\Gamma}$ is bijective and each $\tau_\Gamma$ is injective, thus we have 
\begin{align*}
\phi(\augNor{L_G}{G})& =\phi(\ker \pi_{L_G}) =\phi(\ker(\tau_G \circ \hat\sigma_G \circ \pi_{L_G})) 
=\phi(\ker(\lambda_G \circ \pi_{N_G})) 
=\ker(\lambda_H \circ \pi_{N_H})\\
& =\ker(\tau_H \circ \hat\sigma_H \circ \pi_{L_H})
=\ker(\pi_{L_H})=\augNor{L_H}{H},
\end{align*} as desired.
\end{proof}

\section{Finite $2$-generated $p$-groups with cyclic commutator and $p$ odd} \label{SectionClassification}

In this section \fbox{\emph{$p$ is an odd prime number.}}   We start  by recalling the classification of non-abelian $2$-generated $p$-groups with cyclic commutator subgroup from \cite{OsnelDiegoAngel}. 
Each such group $G$ is showed to be uniquely determined, up to isomorphism, by an integral vector\footnote{The classification in \cite{OsnelDiegoAngel} is performed for all primes $p$  and for $p=2$ the fifth and sixth entries  of $\inv(G)$   may also be $-1$.} 
\[
\inv(G)=(p,m,n_1,n_2,1,1,o_1,o_2,o'_1,o'_2,u_1,u_2)
\] 
of length $12$ whose entries are determined as   described below.  
The first four are straightforward and satisfy:
\begin{itemize}
	\item $|G|=p^m$, 
	\item $G/G'\cong C_{p^{n_1}}\times C_{p^{n_2}}$ with $n_1\ge n_2\ge 1$. 
\end{itemize}
To continue, we define a \emph{basis} of $G$ to be a pair $(b_1,b_2)$ of elements of $G$ such that $G/G'=\GEN{b_1G'}\times \GEN{b_2G'}$ and $|b_iG'|=p^{n_i}$. Let $\B$ denote the set of bases of $G$. Moreover, 
for every $g\in G$, let $o(g)\in\Z_{\geq 0}$ be  such that  $p^{o(g)}$ is the order of $g\Cen_G(G')$ in $G/\Cen_G(G')$, in symbols  $p^{o(g)}=|g\Cen_G(G')|$. Equivalently, $o(g)=m-v_p(r(g)-1)$, where $r(g)$ denotes the unique integer satisfying $2\leq r(g)\leq p^m+1$ and $a^g=a^{r(g)}$ for each $a\in G'$. Define:
\begin{itemize}	
	\item 
		$(o_1,o_2)=\min_{\lex}\{(o(b_1),o(b_2)) : (b_1,b_2) \in \B\}$ and
	\begin{equation}\label{Erres}
	r_1=1+p^{m-o_1} \qand 
	r_2 = \begin{cases} 1+p^{m-o_2}, & \text{if } o_2>o_1; \\ r_1^{p^{o_1-o_2}}, & \text{otherwise}. \end{cases}
	\end{equation}
\end{itemize}
Let now $\B_r=\B_r(G)$ be the set consisting of all bases $(b_1,b_2)$ of $G$ with the property that, for every $a\in G'$ and $i=1,2$, one has $a^{b_i}=a^{r_i}$, equivalently, $r(b_i)\equiv r_i\mod p^m$. The set $\B_r$ is not empty as proved in \cite[Proposition~2.3(5)]{OsnelDiegoAngel}. 
For every $b=(b_1,b_2)\in \B$, let $o'(b)=(o'_1(b),o'_2(b))$ and $u(b)=(u_2(b),u_1(b))$ be defined by 
\begin{equation}
p^{n_i+o'_i(b)}=|b_i|,\quad b_i^{p^{n_i}}=[b_2,b_1]^{u_i(b)p^{m-o'_i(b)}} \qand 1\le u_i(b)<p^{m-o'_i(b)}.
\end{equation}
Define subsequently:
\begin{itemize}
	\item $(o'_1,o'_2) = \max_{\lex}\{o'(b) : b\in \B_r\}$ and
	\item $(u_2,u_1) = \min_{\lex} \{ u(b) : b\in \B_r \text{ with } o ' (b)=(o'_1,o'_2)\}$.  
\end{itemize}
We have described how the entries of $\inv(G)$ are computed directly as structural invariants of $G$.
Conversely, for a list of non-negative integers $I=(p,m,n_1,n_2,o_1,o_2,o'_1,o'_2,u_1,u_2)$, defining $r_1$ and $r_2$ as in \eqref{Erres}, the group $\G_I$ is defined as 
	$$\G_I=\GEN{b_1,b_2,a=[b_2,b_1] \mid a^{p^m}=1, a^{b_i}=a^{r_i}, b_i^{p^{n_i}}=a^{u_ip^{m-o'_i}} \; (i=1,2)}.$$
Denoting by $[G]$ the isomorphism class of a group $G$, the main result of \cite{OsnelDiegoAngel} for $p$ odd takes the following form. 

\begin{theorem}
	\label{Main}
	The maps $[G]\mapsto \inv(G)$ and $I\mapsto [\mathcal{G}_I]$ define mutually inverse bijections between the isomorphism classes of $2$-generated non-abelian groups of odd prime-power  order and the set of lists of integers $(p,m,n_1,n_2,  o_1,o_2,o'_1,o'_2,u_1,u_2)$ satisfying the following conditions. 	
	\begin{enumerate}[label=$(\arabic*)$] 
		\item \label{1} $p$ is prime and $n_1\geq n_2 \ge 1 $.
		\item \label{2}  $0\le o_i<m$, $0\le o'_i \le m-o_i$ and $p\nmid u_i$ for $i=1,2$.    
		\item \label{4} One of the following conditions holds:  
		\begin{enumerate}[label=$(\alph*)$] 
			\item $o_1=0$ and $o'_1\le o'_2\le o'_1+o_2+n_1-n_2$.  
			\item $o_2=0<o_1$, $n_2<n_1$ and $o'_1+\min(0,n_1-n_ 2  -o_1)\le o'_2\le o'_1+n_1-n_2$.
			\item $0<o_2< o_1<o_2+n_1-n_2$ and $o'_1\le o'_2\le o'_1+n_1-n_2$. 
		\end{enumerate}	
		
		\item \label{5}   $o_2+o_1'\leq m\le n_1$  and one of the following conditions hold:
		\begin{enumerate}[label=$(\alph*)$] 
			\item $o_1+o'_2\le m \le n_2$.
			\item $2m-o_1-o'_2=n_2<m$ and $u_2\equiv 1 \mod p^{m-n_2}$.		
		\end{enumerate}    
		
		\item \label{7}$ 1\le u_1  \leq p^{a_1}$, where 
		$a_1=\min(o'_1,o_2+\min(n_1-n_2+o'_1-o'_2,0)).$
		\item \label{8}One of the following conditions holds:
		\begin{enumerate}[label=$(\alph*)$] 
			\item $ 1\le  u_2 \leq p^{a_2}$.
			\item $o_1o_2\neq 0$, $n_1-n_2+o_1'-o_2'=0<a_1$,  $1+p^{a_2}\leq u_2 \leq 2p^{a_2}$, and $u_1 \equiv 1\mod p$;  
		\end{enumerate}
		where  
		$$a_2=  \begin{cases}
		0, &\text{if } o_1=0; \\
		\min(o_1,o'_2,o'_2-o'_1+\max(0,o_1+n_2-n_1)), & \text{if } o_2=0<o_1; \\ \min(o_1-o_2,o'_2-o'_1), & \text{otherwise.} \end{cases}$$
	\end{enumerate}
\end{theorem}

\fbox{\begin{minipage}{6.2in}
\emph{In the remainder of the section, $G$ denotes a finite non-abelian $2$-generated $p$-group with cyclic commutator subgroup, with invariant vector
	$$\inv (G)=(p,m,n_1,n_2, 1,1,o_1,o_2,o_1',o_2',u_1,u_2). $$ 
and associated $r_1$ and $r_2$ as in \eqref{Erres}.}
\end{minipage}}
\vspace{5pt}

Thanks to \Cref{Main}\ref{2} and \Cref{EseProp}\ref{ValEse} we have 
\begin{equation}\label{vpr}
v_p(r_i-1) =m-o_i>0 \text{ for } i=1,2.
\end{equation}
The following two lemmas are Lemma~2.2 and Lemma~4.2 from \cite{OsnelDiegoAngel}. 

\begin{lemma}\label{Fijando-rLema}
	Let $b=(b_1,b_2)$ be a basis of $G$. Then $(o(b_1),o(b_2))=(o_1,o_2)$ if and only if  one of the following conditions holds:
	\begin{enumerate}[label=$(\arabic*)$] 
		\item\label{o10}  $o(b_1)=0$.  
		\item\label{o20}  $0=o(b_2)<o(b_1)$ and  $n_2<n_1$. 
		\item\label{osNo0}  $0< o(b_2) < o(b_1) < o(b_2) +n_1-n_2$. 
	\end{enumerate}
\end{lemma}

\begin{lemma}\label{OPrima++} 
	 Let   $b\in \B_r$. 
	Then $o'(b)=(o'_1,o'_2)$ if and only if the following conditions hold: 
	\begin{enumerate}[label=$(\arabic*)$] 
		\item If $o_1 =0$ then  $o'_1(b)\le o'_2(b)\le o'_1(b)+o_2 +n_1-n_2$.
		
		\item If  $o_2 =0<o_1$ then   $o'_1(b)+\min(0,n_1-n_2-o_1 )\le o'_2(b)\le o'_1(b)+n_1-n_2$.
		\item If $o_1 o_2 \ne 0$ then $o'_1(b)\le o'_2(b)\le o'_1(b)+n_1-n_2 $.
	\end{enumerate}
\end{lemma}

For the following result, recall  from the introduction  that 
\[
O(G)=\min_{\lex} \{ (|b_1\Cen_G(G')|, |b_2\Cen_G(G')|, -|b_1|, -|b_2|) \mid (b_1,b_2)\in \mathcal{B} \}.
\]

\begin{lemma}\label{Ooo'}
The following equality holds: $\OO(G)=(p^{o_1},p^{o_2},-p^{n_1+o'_1},-p^{n_2+o'_2})$. 
\end{lemma}

\begin{proof}
	For every $g\in G$ let $r(g)$ be the unique integer $2\le r(g)\le p^m+1$ such that $a^g=a^{r(g)}$ for every $a$ in $G'$. From \Cref{EseProp}\ref{ValEse} it follows that 
	\begin{equation}\label{o=Order}
	p^{o(g)}=|g\Cen_G(G')|=p^{m-v_p(r(g)-1)}=o_{p^m}(r(g)) \textup{ for all } g\in G.
	\end{equation}
	In particular, if $b\in \B_r$ and $i=1,2$, then $o(b_i)=o_i$.
	Thus the first two entries of $\OO(G)$ are $p^{o_1}$ and $p^{o_2}$. 
 To deal with the remaining two entries,  fix two bases $b=(b_1,b_2)$ and $b'=(b'_1,b'_2)$ of $G$ with $b'\in \B_r$ and  such that $\OO(G)=(|b_1\Cen_G(G')|,|b_2\Cen_G(G')|,-|b_1|,-|b_2|)$  and $o'(b')=(o'_1,o'_2)$. 
 In particular, we have  $|b_i\Cen_G(G')|=|b'_i\Cen_G(G')|=p^{o_i}$. 
	Moreover, $|b_i|=p^{n_i+o'_i(b)}$ and $|b'_i|=p^{n_i+o'_i}$.
	Thus $(-p^{n_1+o'_1(b)},-p^{n_2+o'_2(b)})=(-|b_1|,-|b_2|)\le_{\lex} (-p^{n_1+o'_1},-p^{n_2+o'_2})$ or equivalently $o'(b)\ge_{\lex} (o'_1,o'_2)$. 
	On the other hand, the two automorphisms of $G'$ given by $a\mapsto a^{b_i}$ and $a\mapsto a^{b'_i}$ have order $p^{o_i}$. 
	Since $\Aut(G')$ is cyclic, there exist integers $x_1$ and $x_2$, both coprime to $p$, such that $b''=(b_1^{x_1},b_2^{x_2})\in \B_r$. 
	Thus $p^{n_i+o'_i(b'')} = |b_i^{x_i}|=|b_i|=p^{n_i+o_i'(b)}$ and hence 
	$o'(b)=o'(b'')\le_{\lex} (o'_1,o'_2)$.
	We conclude that $o'(b)=(o'_1,o'_2)$ and hence $\OO(G)=(p^{o_1},p^{o_2},-p^{n_1+o'_1},-p^{n_2+o'_2})$.
\end{proof}

\fbox{\begin{minipage}{6.2in}
\emph{In the remainder of the section let $b=(b_1,b_2)\in\B_r$ be a fixed basis of $G$ such that $o'(b)=(o'_1,o'_2)$ and denote $a=[b_2,b_1]$.} 
\end{minipage}}
\vspace{5pt}

Then the following hold:
\begin{equation}\label{Ordersbi}
	|a|=p^m, \quad |b_iG'|=p^{n_i}, \quad |b_i\Cen_G(G')|=p^{o_i} \quad |b_i|=p^{n_i+o'_i} \qand a^{b_i}=a^{r_i}.
\end{equation}
In particular, every element of $G$ is of the form $b_1^xb_2^ya^z$ for some integers $x,y,z$. 
Moreover, it follows from   \eqref{eq:conj-r} and \eqref{eq:double-conj}
that, for every non-negative integer $e$, one has 
\begin{eqnarray}
	\label{Exp}	(b_1^x b_2^y a^z) ^{p^e} &=& b_1^{xp^e} b_2^{yp^e} a^{ \Ese{r_1}{x} \Ese{r_2}{y} \Te{r_1^x,r_2^y}{p^e} +z\Ese{r_1^x r_2^y }{p^e}  },  \\
	\label{Comma}	[a^e,b_1^xb_2^ya^z]&=&a^{e(r_1^xr_2^y-1)},  \\  
	\label{Commb1}	[b_1^x b_2^y a^z,b_1]&=&  a^{  \Ese{r_2}{y }   + z (r_1-1)  }  , \\
	\label{Commb2}	[b_1^x b_2^y a^z,b_2]&=&  a^{ -\Ese{r_1}{x}r_2^{y}   +  z(r_2-1)}.
\end{eqnarray} 
The next lemma describes some characteristic features of $G$.

\begin{lemma}\label{Characteristic} The following statements hold: 
\begin{enumerate}[label=$(\arabic*)$]
	\item\label{CenterComm} $\ZZ(G)\cap G' = \GEN{a^{p^{\max(o_1,o_2)}}}$.
	\item\label{exponente} $\exp(G) = p^{\max(n_1+o_1', n_2+o_2')}$.
	\item\label{LCSpodd} If $i\ge 2$ then $\gamma_i(G)= \GEN{ a^{p^{ (i-2)(m-\max(o_1,o_2))}}}$ and the class of $G$ is  $1+\E{\frac{m}{m-\max(o_1,o_2)}}$. 
\end{enumerate}
\end{lemma}

\begin{proof} 
\ref{CenterComm}
Let $w$ be a non-negative integer. As $v_p(r_i-1)=m-o_i$ and $a^{b_i}=a^{r_i}$ , we have that $a^{p^w}\in \ZZ(G)$ if and only if, for each $i\in\{1,2\}$, one has $w+m-o_i\ge m$. Then $\ZZ(G)\cap G'=\GEN{a^{p^{\max(o_1,o_2)}}}$.  


\ref{exponente} 
Let $e=\max(n_1+o'_1,n_2+o'_2)$. By \eqref{Ordersbi}, we have that $\exp(G)\geq p^e$ so we show that   $\mho_e(G)=1$. To this end, note that $e\geq m$ as a consequence of \Cref{Main}\ref{5} and thus   $\mho_e(G')=1$. Now regularity yields that $p^e$-th powering induces a homomorphism $G/G'\rightarrow G$ and so, as a consequence of \eqref{Ordersbi}, we get $\mho_e(G)=1$.

\ref{LCSpodd} We work by induction on $i$ and, as the base case $i=2$ is clear, we assume that $i>2$ and the claim holds for $i-1$. In other words, write $f=(i-3)(m-\max(o_1,o_2))$ so that $\gamma_{i-1}(G)=\GEN{a^{p^f}}$. It follows then from \eqref{Comma} that 
\[
\gamma_i(G)=\GEN{[a^{p^f},b_1], [a^{p^f},b_2]}=\GEN{a^{p^f(r_1-1)},a^{p^f(r_2-1)}}=\GEN{a^{p^f\min(r_1-1,r_2-1)}}.
\]
We conclude by computing $v_p(p^f\min(r_1-1,r_2-1))=f+m-\max(o_1,o_2)=(i-2)(m-\max(o_1,o_2))$.
\end{proof}

Since each $\Ese{r_i}{-}$ induces a bijection $\Z/p^m\Z\rightarrow \Z/p^m\Z$ (see \Cref{EseBijective}) there are unique integers  $1\le \delta_1\le p^{o_1}$ and $1\le \delta_2\le p^{o_2}$  satisfying the following congruences:
\begin{eqnarray} 
\Ese{r_2}{\delta_1p^{m-o_1}}&\equiv & 1-r_1\mod p^m \label{EqDelta1},\\
\Ese{r_1}{\delta_2p^{m-o_2}} r_2^{\delta_1p^{m-o_1}} &\equiv&  r_2-1 \mod p^m.\label{EqDelta2}
\end{eqnarray}
Moreover, 
 \eqref{Erres} and \Cref{EseProp}\ref{ValEse} yield that $p$ does not divide $\delta_1\delta_2$.  
\begin{lemma}
The following hold:
\begin{equation}\label{deltas}
\begin{cases} \delta_1=\delta_2=1, & \text{if } o_1=0; \\ \delta_1+\delta_2\equiv 0 \mod p^{o_2}, & \text{otherwise}.\end{cases}
\end{equation}
\end{lemma}

\begin{proof}
Assume first that $o_1=0$, implying that $\delta_1=1$, $r_1=1+p^m$ and $r_2=1+p^{m-o_2}$.  Then \cref{EseProp}\ref{ValEse}-\ref{EseCero} implies 
$$\delta_2p^{m-o_2} \equiv \Ese{r_1}{\delta_2p^{m-o_2}}r_2^{p^{m-o_1}}\equiv r_2-1=p^{m-o_2} \mod p^m$$ 
and hence $\delta_2=1$. 
 Suppose   now  that $o_1>0$, which ensures that $o_1>o_2$ and $r_2=r_1^{p^{o_1-o_2}}$. 
  As a consequence, we have  $v_p(r_2-1)=m-o_2>v_p(r_1-1)=m-o_1$. Moreover, by the definition of the $\delta_i$'s, there are integers $\lambda$ and $\mu$  such that
$\Ese{r_2}{\delta_1p^{m-o_1}}=1-r_1+\lambda p^m$ and $\Ese{r_1}{\delta_2 p^{m-o_2}}r_2^{\delta_1 p^{m-o_1}}=r_2-1+\mu p^m$. Then the following identities hold: 
\begin{eqnarray*}
r_1^{(\delta_1+\delta_2)p^{m-o_2}} -1 & = & 
r_1^{\delta_2p^{m-o_2}}r_2^{\delta_1p^{m-o_1}} -1 \\
&=& (r_1^{\delta_2p^{m-o_2}}-1)r_2^{\delta_1p^{m-o_1}}+r_2^{\delta_1p^{m-o_1}}-1 \\
&=& (r_1-1) \Ese{r_1}{\delta_2p^{m-o_2}} r_2^{\delta_1p^{m-o_1}} + (r_2-1)\Ese{r_2}{\delta_1p^{m-o_1}} \\
 &=&p^m(\mu(r_1-1)+\lambda(r_2-1)) \equiv 0 \mod p^{2m-o_1}.	
\end{eqnarray*}
  We have shown that  $  p^m =o_{p^{2m-o_1}}(r_1)$  divides  $(\delta_1+\delta_2)p^{m-o_2}$ and hence $\delta_1+\delta_2\equiv 0 \mod p^{o_2}$, as desired.  
\end{proof}

\begin{lemma}\label{CenterOdd}
One has	$\ZZ(G)=\begin{cases} \GEN{b_1^{p^m},b_2^{p^m},b_1^{p^{m-o_2}}a}, & \text{if } o_1=0; \\ \GEN{b_1^{p^m},b_2^{p^m},b_1^{-\delta_1p^{m-o_2}}b_2^{\delta_1p^{m-o_1}}a}, & \text{otherwise}. \end{cases}$ 
\end{lemma}

\begin{proof}
Let $g=b_1^x b_2^ya^z$ be an arbitrary element of $G$ with $x,y,z\in \Z$.
We characterize when $g\in\ZZ(G)$ in terms of conditions on the exponents $x,y,z$. 
For this, note that \eqref{Commb1} and \eqref{Commb2} ensure that $g=b_1^x b_2^ya^z\in \ZZ(G)$ if and only if the following congruences hold:
	\begin{eqnarray} 
	\label{EqZ1}	\Ese{r_2}{y}&\equiv & z(1-r_1)\mod p^m ,\\
	\label{EqZ2}	\Ese{r_1}{x} r_2^y &\equiv&  z(r_2-1)\mod p^m.
	\end{eqnarray}
In particular,   the elements $b_1^{p^m}$, $b_2^{p^{m}}$ and $c=b_1^{\delta_2p^{m-o_2}}b_2^{\delta_1p^{m-o_1}}a$ are all central.  
Let 
	$$d=\begin{cases} b_1^{p^{m-o_2}}a, & \text{if } o_1=0; \\
	b_1^{-\delta_1p^{m-o_2}}b_2^{\delta_1p^{m-o_1}}a, & \text{otherwise.} \end{cases}$$
By \eqref{deltas} we have $\GEN{b_1^{p^m},b_2^{p^m},d}=\GEN{b_1^{p^m},b_2^{p^m},c}\subseteq \ZZ(G)$.
If $o_1=0$ and $g=b_1^xb_2^ya^z\in \ZZ(G)$    then it follows from  \eqref{Erres}  that \eqref{EqZ1} is equivalent to $y\equiv 0 \mod p^m$ and hence \eqref{EqZ2} is equivalent to $x\equiv zp^{m-o_2}\mod p^m$. Thus $g\in \GEN{b_1^{p^m},b_2^{p^m},d}$ and hence $\ZZ(G)=\GEN{b_1^{p^m},b_2^{p^m},d}$, as desired.

Suppose otherwise that $o_1>0$ and define $B=\GEN{b_2^{p^{m-o_1}}}$, $N=b_2^{p^{\min(m,n_2)}}$ and $f:\ZZ(G)\rightarrow B/N$ by $f(b_1^xb_2^ya^z)=b_2^y N$. The map $f$ is well defined because, on the one hand if 
$b_1^xb_2^ya^z=b_1^ub_2^va^w$ then $y \equiv v \mod p^{n_2}$ and hence $b_2^yN= b_2^v N$; and on the other hand if $b_1^xb_2^ya^z\in\ZZ(G)$ then  \eqref{EqZ1} ensures that  $m-o_1=v_p(r_1-1)\le  \Ese{r_2}{y}=v_p(y)$  and hence $b_2^y\in B$.

We claim that $\ker f=\GEN{b_1^{p^m},b_2^{p^m},a^{p^{o_1}}}$. The inclusion from right to left is clear.   Assume  $g=b_1^xb_2^ya^z\in \ker f$. If $m\le n_2$   this implies that $p^m$ divides $y$. Then from \eqref{EqZ1}  and \cref{EseProp}\ref{ValEse} we have that $v_p(z)\ge o_1$ and hence from \eqref{EqZ2} we deduce that $v_p(x)\ge m-o_2+o_1>m$. 
 This shows that  $g\in \GEN{b_1^{p^m},b_2^{p^m},a^{p^{o_1}}}$, as desired. 
Suppose  now that $m>n_2$. Then $g=b_1^xa^z$ 
for some integers $x$ and $z$ and hence again from \eqref{EqZ1} we deduce that $v_p(z)\ge o_1$ and from \eqref{EqZ2} we  conclude  that $v_p(x)\ge m-o_2+o_1>m$. Therefore, again $g\in \GEN{b_1^{p^m},b_2^{p^m},a^{p^{o_1}}}$   and the claim  is proven.

We finally show that $\ZZ(G)=\GEN{b_1^{p^m},b_2^{p^m},c}$. To this end, observe that $B/N$ is generated by $f(c)$, as $p$  does not divide  $\delta_1$. 
This together with the claim and the fact that $f$ is a group homomorphism implies that $\ZZ(G)=\GEN{c,\ker f} = \GEN{b_1^{p^m},b_2^{p^m},a^{p^{o_1}},c}$. 
To complete the proof  we show that  $ a^{p^{o_1}}\in \GEN{b_1^{p^m},b_2^{p^m}, c}$. To this end, observe that the group $H=\GEN{b_1^{p^{m-o_2}}, b_2^{p^{m-o_1}},a}$ is regular   and that  $H'=\GEN{a^{p^{2m-o_1-o_2}}}$. Indeed,  $[a,b_1^{p^{m-o_2}}]= a^{r_1^{p^{m-o_2}}-1} $, $[a,b_2^{p^{m-o_1}}]=a^{r_2^{p^{m-o_1}}-1}$  and $[b_2^{p^{m-o_1}}, b_1^{p^{m-o_2}}]=a^{\Ese{r_1}{p^{m-o_2}}\Ese{r_2}{p^{m-o_1}}}$, and these three elements generate the same subgroup $\GEN{a^{p^{2m-o_1-o_2}}}$ since $$v_p(r_1^{p^{m-o_2}}-1)=v_p(r_2^{p^{m-o_1}}-1)=v_p(\Ese{r_1}{p^{m-o_2}}\Ese{r_2}{p^{m-o_1}})=2m-o_1-o_2.$$ 
As $(a^{p^{2m-o_1-o_2}})^{p^{o_1}}=a^{p^{2m-o_2}}=1$, from the regularity of $H$ it follows that $c^{p^{o_1 }}=b_1^{\delta_2p^{m-o_2+o_1}}b_2^{\delta_1 p^m} a^{p^{o_1}}$. Now $o_1>0$ implies $m-o_2+o_1>m$ and so the proof is complete.
\end{proof}

\begin{lemma}\label{G/L-N}
Let $t=m-\max(o_1,o_2)$. Then 
 the following hold:
\begin{enumerate}[label=$(\arabic*)$]
\item\label{ZGG'} $|\ZZ(G)\cap G'|=p^t$, $\ZZ(G)G'=\GEN{a,b_i^{p^m},b_1^{-p^{m-o_2}}b_2^{p^{m-o_1}}}$ and  $G/\ZZ(G)G'\cong C_{p^m}\times C_{p^t}$. 
\item\label{G/CGG'} 
$\Cen_G(G')= \Omega_t(G:\ZZ(G)G')  = \begin{cases} \GEN{a,b_1,b_2^{p^{o_2}}}, & \text{if } o_1=0; \\
\GEN{a,b_1^{p^{o_1}}, b_1^{p^{o_1-o_2}} b_2^{-1}}, & \text{otherwise}.\end{cases}$
\item \label{CGG''} $\Cen_G(G')'=  \mho_{m-t}(G') $.

\item\label{ExpCGG'o1=0} If $o_1=0$ then $\exp(\Cen_G(G'))=p^{n_1+o'_1}$.
\item\label{ExpCGG'o2=0} If $o_2=0$ then $\exp(\Cen_G(G'))=p^{\max(n_1+o'_1-o_1,n_2+o'_2)}$.

\end{enumerate}
\end{lemma}

\begin{proof}
 \ref{ZGG'}  This  is a direct consequence of  \Cref{Characteristic}\ref{CenterComm} and \Cref{CenterOdd}. 
 
\ref{G/CGG'}  Let $g=b_1^xb_2^ya^z$ be an arbitrary element of $G$. Then \eqref{Comma} yields that $g\in\Cen_G(G')$ if and only if
$a^{r_1^xr_2^y-1}=1$. 
If $o_1=0$ then the last equality holds if and only if $(1+p^{m-o_2})^y-1\equiv 0\bmod p^m$, equivalently if $p^{o_2}$ divides $y$. 
So 
\begin{equation}\label{CGG'1}
o_1=0 \ \ \textup{implies} \ \ \Cen_G(G')=\GEN{a,b_1,b_2^{p^{o_2}}}.
\end{equation}
Suppose now that $o_1>0$. 
Then $r_1=1+p^{m-o_1}$ and $r_2=r_1^{p^{o_1-o_2}}$ yielding that $[a,g]=a^{(1+p^{m-o_1})^{x+yp^{o_1-o_2}}-1}$.
As $o_{p^m}(1+p^{m-o_1})=p^{o_1}$, we have that $g\in\Cen_G(G')$ if and  only if $x+yp^{o_1-o_2} \equiv 0 \bmod p^{o_1}$, that is there exists an integer $v$ with $x=-yp^{o_1-o_2}+vp^{o_1}$. We have proven that 
\begin{equation}\label{CGG'2}
o_1>0 \ \ \textup{implies} \ \ \Cen_G(G')=\GEN{a,b_1^{p^{o_1}},b_1^{p^{o_1-o_2}}b_2^{-1}}.
\end{equation}
To conclude, let $L=\mho_t(G:Z(G)G')$. By \ref{ZGG'}, \eqref{CGG'1} and \eqref{CGG'2} we have that $\Cen_G(G')\subseteq L$ and $G/\Cen_G(G')\cong G/L\cong C_{p^{m-t}}$ and therefore $\Cen_G(G')=L$.

\ref{CGG''} 
The group $G'$ being cyclic, $\mho_{m-t}(G')=\Omega_t(G')$ and so \ref{G/CGG'} and the regularity of $G$ yield that 
$\mho_t(L')= [L,\mho_t(L)]\subseteq [L,\ZZ(G)G']=1$,  that is  $L'\subseteq \Omega_t(G')=\mho_{m-t}(G')$.
For the other inclusion, it suffices to observe that 
\[
\mho_{m-t}(G')=\GEN{a^{p^{m-t}}}=
\begin{cases}
\GEN{[b_1,b_2^{p^{o_2}}]}, & \textup{ if } o_1=0;\\
\GEN{[b_1^{p^{o_1}},(b_1^{p^{o_1-o_2}}b_2^{-1})]}, & \textup{ otherwise.}
\end{cases}
\]

\ref{ExpCGG'o1=0}-\ref{ExpCGG'o2=0}
Assume that $o_1=0$. Then we have $b_1\in \Cen_G(G')$ and $|b_1|=p^{n_1+o'_1}$, from which we derive $\exp(G)\ge p^{n_1+o'_1}$. Let now $e=n_1+o'_1$. Then \Cref{Main}\ref{4}-\ref{5} yields that $m\le e$ and $n_2+o'_2-o_2\le e$.
It follows from \eqref{Ordersbi}, \eqref{Exp} and \Cref{EseProp}\ref{ValEse}-\ref{ValEse2} that $(b_1^xb_2^{yp^{o_2}}a^z)^{p^e}=1$ for every $x,y,z\in \Z$. We have shown that $\exp(\Cen_G(G'))=p^{n_1+o'_1}$. A similar argument  works when   $o_2=0<o_1$.  
\end{proof}

In the following lemma, $\Soc(G')$ denotes the \emph{socle} of $G'$. 
We remark that $\Soc(G')=\GEN{a^{p^{m-1}}}$, because $G'$ is cyclic of order $p^m$ and $m\ge 1$.

\begin{lemma}\label{ParametersQuotient}
Write $\overline G= G/\Soc(G')$ and assume that $m\geq 2$. 
Then $\overline{G}$ is a non-abelian group and one has 
$\inv(\ov G)=(p,m-1,n_1,n_2,1,1,\overline o_1, \overline o_2, \overline o'_1,\overline o'_2, \ov u_1,\ov u_2)$ where:
\begin{align*}
\overline o_i &=\max (0, o_i -1), \hspace{0.60cm} \text{for }i=1,2; \\ 
\overline o_1'  &=  \begin{cases}
o_1' ,&\text{if }   o_1 =0,  \ 0<\min(o_1' ,o_2 ) \text{ and }o_2'  = o_1'  +o_2  +n_1 -n_2  ;\\
\max(0,o_1' -1),&\text{otherwise};
\end{cases} \\ 
\overline o_2'  &= \begin{cases}
o_2' , &\text{if }  o_2 =0 , \ n_1 -n_2 <o_1  \text{ and }0<o_2 ' =o_1'  +n_1 -n_2 -o_1   ; \\
\max(0,o_2' -1), &\text{otherwise}.
\end{cases} 
\end{align*}

\end{lemma}

\begin{proof} 
That the first four entries of $\inv(\overline G)$ are $p, m-1, n_1$ and $n_2$ is obvious.
Since we are dealing with two groups $G$ and $\overline G$, in this proof we distinguish $\B_{r}=\B_r(G)$ and $\B_r(\overline G)$. 

Let $\overline g$ denote the natural image in $\overline G$ of an element $g\in G$. Then $r(\overline g)$ is the unique integer in the interval $[2,p^{m-1}+1]$ that is congruent to $r(g)$ modulo $p^{m-1}$. 
By \eqref{o=Order}   we have 
	$$m-1 -o(\overline{b_i})= v_p(r(\overline{b_i})-1)=  m-\max(1,o_i). $$
Hence $o(\overline{b_i})=\max(0,o_i-1)$ and thus \Cref{Fijando-rLema} yields   
that $\overline o_i=\max(0,o_i-1)$.  
Moreover,   $\bb=(\bb_1,\bb_2)$ is an element of $ \B_r(\overline G)$.    
Note that the following hold:
$$b_i^{p^{n_i+o_i'  }}=1; \quad \overline b_i^{p^{n_i +o_i' -1}}= \begin{cases}
	\overline a^{p^{m-1}}=1, & \text{if } o_i' \neq 0;  \\
	\bb_i^{p^{n_i-1}}\neq 1 & \text{if }o_i'  =0;
	\end{cases}   \qand   	\bb_i^{p^{n_i+o_i' -2}}=\begin{cases} 
	\overline a^{p^{m-2}}\neq 1, & \text{if } o_i' >1; \\
	\bb_i^{p^{n_i-1}}\neq 1, & \text{if }o_1' =1.  
	\end{cases} $$ 
	Therefore  we derive that 
	$o_i'(\bb)=\log_p |\bb_i|-n_i = \max(0,o_i' -1).$ 

To finish the proof we  distinguish  three cases  according to \Cref{OPrima++} and   search for some $\hat{b}\in \B_r(\overline{G})$ satisfying the corresponding conditions in the lemma.   Then   \Cref{OPrima++} will guarantee  that $o'_i=o'_i(\hat b_i)$ for $i=1,2$.
In most cases $\overline{b}$  already  satisfies the desired conditions and hence, in such cases, we  take     $\hat{b}=\bb$ and hence $\overline o'_i= o_i'(\bb)=\max(0,o'_i-1)$.
Otherwise we modify slightly $\overline{b}$ to obtain the desired $\hat{b}$.  

\textbf{Case 1}. Suppose first that $o_1=0$.   By \cref{Main}\ref{4}  we have 
$$o'_1\le o'_2\le o'_1+o_2+n_1-n_2$$ and hence $o'_1(\overline b)=\max(0,o'_1-1)\le \max(0,o'_2-1)\le \overline o'_2(\overline b)$. 
Moreover, unless $0<\min(o'_1,o_2)$ and $o'_2=o'_1+o_2+n_1-n_2$, we also have $o'_2(\overline b) \le o'_1(\overline b)+\overline{o_2}+n_1-n_2$. As $\overline{o}_1=0$, the conditions in \Cref{OPrima++} hold for $\hat{b}=\bb$ and hence we have $\overline o'_i=o'_i(\bb)=\max(0,o'_i-1)$, as desired. 
  Assume now that  $0<\min(o'_1,o_2)$ and $o'_2=o'_1+o_2+n_1-n_2$. Then $o'_1(\overline b)=o'_1-1$, $o'_2(\overline b)=o'_2-1$ and $\overline o_2=o_2(\overline b)=o_2-1$ and hence $\overline b$ does not satisfy the hypotheses of \Cref{OPrima++}. 
Then we take $\hat b=(\bb_1 \bb_2^{p^{\overline o_2}},  \bb_2)$, which belongs to $\B_r(\overline G)$ because $[\overline{b_2}^{p^{\overline{o_2}}},\overline{a}]=1$. 
Using \eqref{Exp}, the regularity of $G$ and $m\le n_1$ we  compute 
\[
\hat b_1^{p^{n_1}} = \overline b_1^{p^{n_1}}\overline b_2^{p^{n_1+\overline o_2}}=\overline a^{p^{m-o'_1}} \overline b_2^{p^{n_1+o_2-1}}=\overline a^{p^{m-o'_1}}     \overline b_2^{p^{n_2+o'_2-o'_1-1}} = \overline a^{p^{m-o'_1}+p^{m-o'_1-1}}
\] 
and hence $|\hat b_1|=p^{n_1+o'_1}$ so that $o'_1(\hat b)=o'_1$, $o'_2(\hat b)=o'_2-1$ and we conclude  from \cref{OPrima++} that $\overline o_1' = o'_1$ and $\overline o'_2=o'_2-1$. 
This yields the desired conclusion because in this case $n_1-n_2\ge 0=o_1$ and $o'_2\ge o'_1>0$.  

\textbf{Case 2}. Suppose that $o_2=0<o_1$ so that \cref{Main}\ref{4} ensures $$o'_1+\min(0,n_1-n_2-o_1)\le o'_2\le o'_1+n_1-n_2.$$ 

Assume first that $n_1-n_2\ge o_1$. Then we have $n_1-n_2\ge \overline o_1$ and $o'_1\le o'_2\le o'_1+n_1-n_2$ and consequently also $o'_1(\overline b)\le o'_2(\overline b)\le o'_1(\overline b)+n_1-n_2$. Hence $\overline o'_i=o'_i(\overline b)=\max(0,o'_i-1)$.  

Suppose now that $n_1-n_2<o_1$. Then by \Cref{Main}\ref{4} we have $0<n_1-n_2$. If $o'_2=0$ or $o'_2>o'_1+n_1-n_2-o_1$ then we also have $o'_1(\overline b)+\min(0,n_1-n_2-\overline o_1)\le o'_2(\overline b)\le o'_1(\overline b)+n_1-n_2$ and again we have $\overline o_i=o'_i(\overline b)=\max(0,o'_i-1)$. 
Assume now that $0<o'_2=o'_1+n_1-n_2-o_1$. It follows that $o'_1>0$ and hence $o'_1(\overline b)+\min(0,n_1-n_2-o_1(\overline b))= o'_1 + \min(0,n_1-n_2-o_1)>o'_2(b)-1=o'_2(\overline b)$. Therefore \Cref{OPrima++} yields $o'(\overline b)\ne \overline o'$. 
In this case we take the basis $\hat b=(\overline b_1 ,\overline b_1^{p^{\overline o_1}}\overline b_2)$, which again belongs to the set $\B_r(\overline G)$ because $[\overline b_1^{p^{\overline o_1}},\overline a]=1$.
Then $o'_1(\overline b)=o'_1-1$ and if $k=\Ese{r_1}{p^{o_1-1}}\Ese{r_2}{1}\Te{r_1^{p^{o_1-1}},r_2}{p^{n_2}} + \Ese{r_1^{p^{o_1-1}}r_2}{p^{n_2}}$ then \Cref{EseProp} and \Cref{Main}\ref{5} imply $v_p(k)\ge n_2 \ge m+o'_2$. Using \eqref{Exp} we get
	$$(\overline b_1^{p^{o_1-1}}\overline b_2)^{p^{n_2}} = a^{p^{ m-1-o'_2}    +p^{m-o'_2}+k}.$$ 
Therefore  $|\overline b_1^{p^{o_1-1}}\overline b_2|=p^{n_2+o'_2}$ and we conclude that 
$o'_1(\hat b)=o'_1-1$, $o'_2(\hat b)=o'_2$ and hence $\overline o'_1=o'_1-1$ and $\overline o'_2=o'_2$. 
	
\textbf{Case 3}. Finally, suppose that $o_1o_2\ne 0$. Then $\overline o_1=o_1-1>o_2-1=\overline o_2\ge 0$ and \cref{Main}\ref{4} guarantees $o'_1\le o'_2\le o'_1+n_1-n_2$. Hence $o'_1(\overline b)+\min(0,n_1-n_1-\overline o_1) \le o'_1(\overline b)\le o'_2(\overline b)\le o'_1(\overline b)+n_1-n_2$ and we get $\overline o'_i=o'_i(\overline b)=\max(0,o'_i-1)$.
\end{proof}

\section{Proofs of the main results}

In this section we prove \Cref{ExponentCGG'}, \Cref{OTheorem} and \Cref{oo'Theorem}.
The first one will be included in \Cref{JenningsExpCGG'}, which relies on \cref{lemmaIso}. \Cref{OTheorem} is proven shortly after \Cref{JenningsExpCGG'}, while \cref{oo'Theorem} is a consequence of \cref{lem:o=o,lem:o'=o'}. We conclude the section by proving \cref{Cor1.0} and \cref{Cor1.1}, here presented as \cref{cor:Cent} and \cref{cor:type}, respectively.

\begin{lemma}\label{lemmaIso}
 Let $p$ be an odd prime  and let  $G$   be  a finite $p$-group with cyclic commutator subgroup.  Write, moreover, $|G'|=p^m$ and $|G'\cap \ZZ(G)|=p^t$. Then the following hold:  
	$$\Cen_G(G')=\{g\in G : g^{p^t} \in \ZZ(G)G'\} \ \ \textup{ and } \ \ \Cen_G(G')'= \mho_{m-t}(G').$$ Moreover, for every subgroup $N$ of $G$ contained in $  \Cen_G(G')$ one has $\exp(N)= p^{\min \{ n \ :  \ \M_{p^n}(N)=1 \} }$.
 
\end{lemma}

\begin{proof}
 The abelian case is straightforward so we assume that $G'\neq 1$. 
By a Theorem of Cheng \cite{Cheng1982} the group $G$ can be expressed as a central product
	$$G= H*G_1*\dots *G_s*A, $$
where each $G_i$ is a   $2$-generated  group  of nipotency class $2$, the group $H$ is  $2$-generated with $H'=G'$ and $A$ is abelian. 
 For each $i=1,\dots,s$, write $G_i=\GEN{x_i,y_i}$ and $|G_i'|=p^{m_i}$.
Set $K=G_1*\dots *G_s*A$. 
As $G'=H'$ and $[H,G_i]=[H,A]=1$ it follows that $Z(H)\cap H'=Z(G)\cap G'$ and $\Cen_G(G')=\Cen_G(H')=\Cen_H(H')K$.
Let $L =\{g\in G : g^{p^t}\in \ZZ(G)G'\}$  and note that  $A\subseteq L $.
  Moreover, since each $G_i$ is of class $2$, we have that $Z(G_i)=\GEN{G'_i,x_i^{p^{m_i}},y_i^{p^{m_i}}}$ and hence
	$$\ZZ(G)=\GEN{ \ZZ(H), A, x_i^{p^{m_i}}, y_i^{p^{m_i}}\ :\ i=1,\dots,s}.$$ 
For each $i$, observe that $m_i\leq t \leq m$ because $G_1'*\cdots*G_s'$ is contained in $\ZZ(G)\cap G'$. 
Therefore,   for each choice of $i$, one has $\mho_t(G_i)\subseteq \ZZ(G)$ and we derive that $K\subseteq L$. 
Moreover, it follows from $[H,K]=1$ that $Z(H)H'\subseteq Z(G)G'$. Since $|\ZZ(H)\cap H'|=p^t$,  \Cref{G/L-N}\ref{G/CGG'} yields  $\Cen_G(G')=\Cen_H(H')K=\{h\in H : h^{p^t}\in \ZZ(H)H'\}K \subseteq L$.     For the other inclusion take $g=hk\in L$ with $h\in H$ and $k\in K$ and note that $h\in H\cap L\subseteq \Cen_H(H')$ by \Cref{G/L-N}\ref{G/CGG'}.   This shows that $\Cen_G(G')=L$. 

 We now show that $\Cen_G(G')'=\mho_{m-t}(G')$. For this, let $g,h\in C_G(G')$. Then $h ^{p^{t}}\in G'\ZZ(G)$   and, as $C_G(G')$ has nilpotency class $2$, we have that $[g,h]^{p^t}=[g,h^{p^t}]=1$.   We have proven that  $C_G(G')'\subseteq     \mho_{m-t}(G') $   while  \Cref{G/L-N}\ref{CGG''}  ensures that  $\mho_{m-t}(G')=\mho_{m-t}(H')=\Cen_H(H')'\subseteq \Cen_G(G')'$.  

 Finally let $N$ be a subgroup of  $G$  such that $N\subseteq \Cen_G(G')$. Then $N$ has nilpotency class $2$, so that \eqref{EqJenningsLazard} yields $\M_{p^n}(N)=\mho_n(N)$ and the result follows .    
\end{proof} 

 The following result is a stronger version of  \Cref{ExponentCGG'}.

\begin{theorem}\label{JenningsExpCGG'}
	Let $k$ be a field of odd characteristic $p$ and let $G$ be a finite  $p$-group with cyclic commutator subgroup. If $H$ is another group with $kG$ and $kH$ isomorphic as $k$-algebras then \begin{enumerate}[label=$(\alph*)$]
		\item For every  algebra  isomorphism  $\phi:kG\to k H$ preserving augmentation one has that 
		$$\phi(\augNor{\Cen_G(G')}{G})=\augNor{\Cen_H(H')}{H}; $$
		\item\label{JenningsCGG'} $\M_i(\Cen_G(G'))/\M_{i+1}(\Cen_G(G'))\cong \M_i(\Cen_H(G'))/\M_{i +1}(\Cen_H(H'))$ for every $i\ge 1$; 
		\item $\exp(\Cen_G(G'))=\exp(\Cen_H(H'))$;
		\item  $\Cen_G(G')/G'\cong \Cen_H(H')/H'$; 
		\item  $\Cen_G(G')/\Cen_G(G')'\cong \Cen_H(H')/\Cen_H(H')'$.   
	\end{enumerate}  
\end{theorem}
 
\begin{proof}  
 Let $H$ be a group such that $kG\cong kH$ and let $\Gamma\in\{G,H\}$. It is well known that there  exists an isomorphism $kG\rightarrow kH$  preserving augmentation  (see e.g.\ the remark   on page 63 of \cite{Seh78}).
 Then $H'$ is also cyclic  as a consequence of \Cref{Known}\ref{it:2K}\ref{Generators}.
Moreover, the  number  $|\ZZ(G)\cap G'|=p^t$ is an algebra invariant of $kG$ by \Cref{Known}\ref{it:2K}\ref{Z/ZG'}. 
	By \Cref{lemmaIso}  and  \Cref{Known}\ref{ZGG'Can}, the  hypotheses  of \Cref{mainlemma} hold for $L_{\Gamma}=\Cen_{\Gamma}(\Gamma')$ and $N_{\Gamma}=Z(\Gamma)\Gamma'$. 
	Therefore, if $\phi:kG \rightarrow kH$ is an algebra isomorphism then   $\phi(\augNor{\Cen_G(G')}{G})=\augNor{\Cen_H(H')}{H}$ and hence $\M_i(\Cen_G(G'))/\M_{i+1}(\Cen_G(G'))\cong \M_i(\Cen_H(G'))/\M_i(\Cen_H(H'))$ by \Cref{JenningsNor}. 
	This implies that the  lists of orders of the terms  of the Jennings series of $\Cen_G(G')$ and $\Cen_H(H')$ are equal and, by \Cref{lemmaIso}, these groups have the same exponent.  Finally, observe that  since $\augNor{\Cen_G(G')}{G} /\augNor{G'}{G} \cong  \augNor{\Cen_G(G')/G'}{ G/G' }$, if $\hat \phi:k(G/G')\to k(H/H')$ is the isomorphism induced by $\phi$ then $\hat\phi(\augNor{\Cen_G(G')/G'}{G/G'} )=\augNor{\Cen_H( H')/H'}{H/H'}$. This and the fact that the groups $\Cen_G(G')/G'$ and $\Cen_H(H')/H'$ are both abelian yield, by using the argument in \cite[Lemma~14.2.7(ii)]{Pas77}, that they are isomorphic. 
Writing $p^m=|G'|$, \Cref{lemmaIso} ensures that $\augNor{\Cen_G(G')'}{G}=\augNor{G'}{G}^{p^{m-t}}$ and so $\phi$ induces another isomorphism $\tilde\phi: k(G/\Cen_G(G')')\to k(H/\Cen_H(H')')$, and the same argument yields that $\Cen_G(G')/\Cen_G(G')'\cong \Cen_H(H')/\Cen_H(H')'$. 
\end{proof}

In the remainder of the section we prove \Cref{OTheorem}, \Cref{oo'Theorem}, \Cref{Cor1.0} and \Cref{Cor1.1}.
For that we fix a field $k$ of odd characteristic $p$, a finite $2$-generated  $p$-group $G$ with cyclic commutator subgroup and a group $H$ such that $kG\cong kH$. 
Then, by \Cref{Known}\ref{it:2K}\ref{Generators}, the group $H$ is also $2$-generated with a cyclic commutator subgroup. Moreover, \cref{Known}\ref{it:2K}\ref{it:MIPAbelian} yields that, if one of the two groups is abelian, then $G\cong H$. We assume hence without loss of generality that $G$ and $H$ are non-abelian. 
Now, by \Cref{Known}\ref{it:2K}\ref{AbelDeter}, the first six entries of $\inv(G)$ and $\inv(H)$ coincide. 
Thus we set 
\begin{equation}\label{eq:invy}
\inv (\Gamma)= (p,m ,n _1, n_2,1,1,o_1^\Gamma, o_2^\Gamma , o_1'^\Gamma,o_2'^\Gamma,u_1^\Gamma,u_2^\Gamma) \quad \text{for } \Gamma\in \{G,H\}
\end{equation}
and observe that $m\geq 1$.
To simplify the notation we denote $o^\Gamma=(o_1^{\Gamma},o_2^{\Gamma})$ and $o'^\Gamma=(o_1'^{\Gamma},o_2'^{\Gamma})$. 
We will prove that $o^G=o^H$ and $o'^G=o'^H$ in \cref{lem:o=o} and \cref{lem:o'=o'}, respectively.
  These results will imply  \Cref{oo'Theorem}.  
Combining  \cref{oo'Theorem}   with \Cref{Ooo'}, we  will obtain \Cref{OTheorem} .

For the proofs of \cref{lem:o=o,lem:o'=o'}, we argue by induction on $m$.   The induction base case is covered by \cref{Known}\ref{it:3K}\ref{2GenClass2}, as well as the case where $G$ has nilpotency class $2$. Because of this, we  assume without loss of generality that both $G$ and $H$ are of nilpotency class greater than $2$ and so $m\ge 2$:  \Cref{Characteristic}\ref{LCSpodd} implies that $o^G\ne (0,0)$ and $o^H\ne (0,0)$. 
 Additionally we assume that, if $\tilde{G}$ is a $2$-generated finite $p$-group with cyclic commutator subgroup of cardinality $|\tilde{G}'|<p^m$, then $(o^{\tilde{G}},o'^{\tilde{G}})$ is an algebra invariant of $k \tilde{G}$. 
Denote now
$$\overline \Gamma=\Gamma/\Soc(\Gamma'), \quad \inv(\overline \Gamma) =
(p,m-1,n_1,n_2,1,1,o_1^{\overline \Gamma},o_2^{\overline \Gamma},o_1'^{\overline \Gamma} ,o_2'^{\overline \Gamma},u_1^{\overline \Gamma}, u_2^{\overline \Gamma}), \quad 
o^{\overline \Gamma}=(o_1^{\overline \Gamma},o_2^{\overline \Gamma}), \quad 
o'^{\overline \Gamma}=(o_1'^{\overline \Gamma},o_2'^{\overline \Gamma}).$$
By \cref{Known}\ref{ZGG'Can}, if $\phi:kG \rightarrow kH$ is an isomorphism of algebras, then $\phi(\augNor{G'}{G})=\augNor{H'}{H}$  and, as $\augNor{G'}{G}^{p^{m-1}} = \aug{G'}^{p^{m-1}} kG=\aug{\Soc(G') } kG=\augNor{\Soc(G')}{G}$, it follows that $\phi(\augNor{\Soc(G')}{G})=\augNor{\Soc(H')}{G}$. We derive that
\begin{equation*}
k\overline G\cong \frac{kG}{\augNor{\Soc(G')}{G}}\cong \frac{kH}{\augNor{\Soc(H')}{H}} \cong k\overline H
\end{equation*} 
and the induction hypothesis yields that 
\begin{equation}\label{oo'bar}
o^{\overline G}=o^{\overline H} \qand o'^{\overline G}=o'^{\overline H}.
\end{equation}
As in the previous section, fix $b=(b_1,b_2)\in \B_r(\Gamma)$ with $o'(b)=o'^{\Gamma}$ for $\Gamma\in \{G,H\}$. Thus $b_1, b_2$ and $a=[b_2,b_1]$ have different meanings depending on whether they are considered as elements in $G$ or $H$. The context, however, shall always be clear and any confusion avoided.

\begin{lemma}\label{lem:o=o}
One has $o^G=o^H$.
\end{lemma}

\begin{proof}
	By means of contradiction assume $o^G\neq o^H$ and without loss of generality suppose $o^G<_{\lex} o^H$. 
	Recall that $\ZZ(G)\cap G'\cong \ZZ(H)\cap H'$, by \Cref{Known}\ref{it:2K}\ref{Z/ZG'}, so \Cref{Characteristic}\ref{CenterComm} implies $\max\{o_1^G,o_2^G\}=\max\{o_1^H,o_2^H\}$.   Combining this with \Cref{Main}\ref{4} it follows that $n_2<n_1$, for otherwise $o_1^G=o_1^H=0$ and the previous maximum equals $o_2^G=o_2^H$.  
	As a consequence of \Cref{ParametersQuotient} and \eqref{oo'bar} we get that 
	$$\max(0,o_i^G-1)=o_i^{\overline G}=o_i^{\overline H}=\max(0,o_i^H-1).$$
	Then either $\{o_1^G,o_1^H\} =\{0,1\}$ or $\{o_2^G,o_2^H\}= \{0,1\}$, equivalently either $(o_1^G,o_1^H)=(0,1)$ or $o_1^G=o_1^H$ and $(o_2^G,o_2^H)=(0,1)$. 
	Moreover, by  \Cref{Main}\ref{4},  we have that 
	for each $\Gamma\in \{G,H\}$ either
	$o_1^\Gamma=0$, or $o_2^\Gamma=0<o_1^\Gamma$ or $0<o_1^\Gamma-o_2^\Gamma< n_1-n_2$.  
	Hence one of the following conditions holds: 
\begin{enumerate}
 \item $o^G=(0,1)$ and  $o^H=(1,0)$; 
 \item $o^G=(o_1,0)$ and $o^H=(o_1,1)$ with $1<o_1 <1+n_1-n_2$. 
\end{enumerate}	
	We will prove in both cases that $\ZZ(G)G'/ G'$ and $\ZZ(H)H'/H'$ have different exponent which, in view of \Cref{Known}\ref{it:2K} \ref{Z/ZG'}, is not compatible with $kG\cong kH$.
	
(1) First assume $o^G= (0,1)$ and $o^H=(1,0)$. 
 Then, applying    \Cref{CenterOdd}, we have 
$$\ZZ(G)=\GEN{ b_1^{p^m}, b_2^{p^m}, b_1^{p^{m-1}}a} {\qand}  	\ZZ(H)= \GEN{ b_1^{p^m}, b_2^{p^m}, b_2^{   - p^{m-1}} a } .$$
It follows in particular that
\begin{align*}
	\ZZ(G)/\ZZ(G)\cap G'\cong \ZZ(G)G'/G'&=\GEN{b_1^{p^{m-1}}G'}\times \GEN{b_2^{p^m}G'}\cong \Cen_{p^{n_1-m+1}}\times \Cen_{p^{n_2-m}}, \\
	\ZZ(H)/\ZZ(H)\cap H'\cong \ZZ(H)H'/H'&=\GEN{ b_1^{p^m}G'}\times \GEN{b_2p^{m-1}G'}\cong \Cen_{p^{n_1-m}}\times \Cen_{p^{n_2-m+1}},
\end{align*}
and, as $n_2<n_1$, the exponent of $\ZZ(G)/\ZZ(G)\cap G'$ is $p^{n_1-m+1}$ while the exponent of  $\ZZ(H)/\ZZ(H)\cap H'$ is $p^{n_1-m}$. 

(2)  Assume  $o^G= (o_1,0)$ and $o^H=(o_1,1)$ with $1<o_1<1+n_1-n_2$.
It follows from \cref{CenterOdd} that 
$$
	\ZZ(G)=  \GEN{   b_1^{p^m}, b_2^{p^m},  b_2^{-p^{m-o_1}} a}  {\qand}
	\ZZ(H)=  \GEN{ b_1^{p^m}, b_2^{p^m}, b_1^{p^{m-1}} b_2^{-p^{m-o_1}} a  }
$$
and therefore we also have
\begin{align*}
   \ZZ(G)/\ZZ(G)\cap G'\cong\ZZ(G)G'/G'& = \GEN{b_1^{p^m}G' }\times \GEN{b_2^{p^{m-o_1}}G'}\cong \Cen_{p^{n_1-m}}\times \Cen_{p^{n_2-m+o_1}}, \\
 \ZZ(H)/\ZZ(H)\cap H'\cong\ZZ(H)H'/H'&= \GEN{b_1^{  p^{m-1}} b_2^{p^{m-o_1}} G'} \times \GEN{b_2^{p^{m-o_1+1}}G'}\cong  \Cen_{p^{n_1-m+1}} \times \Cen_{p^{n_2-m+o_1-1}}.
\end{align*} 
As $o_1\leq n_1-n_2$, the exponent of $\ZZ(G)/\ZZ(G)\cap G'$ is $n_1-m$, and the exponent of $\ZZ(H)/\ZZ(H)\cap H'$ is $n_1-m+1$. This concludes the proof. 
\end{proof}

In light of \cref{lem:o=o}, until the end of the section we write $o^G=o^H=(o_1,o_2)$.  

\begin{lemma}\label{lem:o'=o'}
One has $o'^G=o'^H$.
\end{lemma}

\begin{proof} 
By means of contradiction we assume that $o'^G\ne o'^H$ and without loss of generality we also assume that $o'^G<_{\lex}o'^H$. 
In particular $G\not\cong H$ and hence \Cref{Known}\ref{it:3K}\ref{Metacyclic} yields that $G$ and $H$ are not metacyclic. It follows that $\max(o_1'^G,o_2'^G,o_1'^H,o_2'^H)< m$ and, as a consequence of \Cref{ParametersQuotient} and \eqref{oo'bar}, one of the following holds: 
\begin{enumerate}
\item $o_2'^H=o_2'^G+1$ and exactly one of the following holds:
\[
(a)\ o_2'^G=0, \quad  (b)\ o_2=0, \ n_1-n_2<o_1  
\text{ and } o_2'^G=o_1'^G+n_1-n_2-o_1.\]
\item $o_1'^H=o_1'^G+1$ and exactly one of the following holds:
\[
(a)\ o_1'^G=0, \quad (b) \ o_1=0, \ 0<\min(o_1'^G,o_2)\text{ and } o_2'^G=o_1'^G+o_2+n_1-n_2.
\]
\end{enumerate}

\underline{Claim}: We can write $\{i,j\}=\{1,2\}$ with $0=o_i<o_j$, $o_j'=o_j'^G=o_j'^H$ and $o_i'^G+1=o_i'^H$. Moreover, if $i=2$, we have that $n_1-n_2<o_1$.

We prove the claim separately for cases (1) and (2). More precisely, in case (1) we will prove the claim with $i=2$ and $j=1$ and in case (2) we prove the claim with $i=1$ and $j=2$.

We first show that $o_1'^G=o_1'^H$. By \Cref{Known}\ref{it:2K}\ref{expDet} we have $  \exp(G)=\exp(H)$ and so \Cref{Characteristic}\ref{exponente} yields that $$\max(n_1+o_1'^G,n_2+o_2'^G)=\max(n_1+o_1'^H,n_2+o_2'^H)= \max(n_1+o_1'^H,n_2+o_2'^G+1).$$  This implies that such maximum is $n_1+o_1'^G=\max(n_1+o_1'^H,n_2+o_2'^G+1)\geq n_1+o_1'^H$. As $o'^G<_\lex o'^H$, it follows that $o_1'^G=o_1'^H$. 
Thus, if the claim fails, one necessarily has $o_2'^G=0$ and hence $2m-o_1+o_2'^G=2m-o_1>m$. From \Cref{Main}\ref{5} we thus derive  $m\le n_2$. If $o_1'^G>0$ then \Cref{Main}\ref{4} yields $0=o_2<o_1$ and $n_1-n_2-o_1<0$. So, when $o_1'^G>0$, the claim follows. 
Otherwise, that is if $o_1'^G=o_2'^G=0$, we have $p^{n_2}\le p^{n_1}=\exp(G)=\exp(H)=\max(p^{n_1},p^{n_2+1})$ and therefore $n_1\geq n_2+1$.
Hence, using \Cref{Characteristic}\ref{LCSpodd} and the fact that $(p^k-2) (m-\max(o_1,o_2)) +n_2-k \geq n_2 \geq m $ for $k\geq 1$, we have
	\begin{equation*}
	\begin{array}{ccc}
	\M_{p^{n_2}}(G)&=& \GEN{b_1^{p^{n_2}}},  \\
	\M_{p^{n_2 }+1}(G)&=& \GEN{b_1^{p^{n_2+1}}},
	\end{array} \qquad \begin{array}{ccc}
	\M_{p^{n_2}}(H)&= &  \GEN{b_1^{p^{n_2}}, a^{p^{m-1}}}, \\
	\M_{p^{n_2 }+1}(H)&=&\GEN{b_1^{p^{n_2+1}} }. 
	\end{array} 
	\end{equation*}
	We obtain the following contradiction to \cref{Known}\ref{it:2K}\ref{JenningsDet}:
	\begin{equation*}
	\Cen_p \cong  \frac{\M_{p^{n_2}}(G)}{ \M_{p^{n_2 }+1}(G)}\cong \frac{\M_{p^{n_2}}(H)}{ \M_{p^{n_2 }+1}(H)} \cong \Cen_p\times \Cen_p  .
	\end{equation*}  
	
Suppose now that (2) holds. If $o_2'^G\neq o_2'^H$ then, by \Cref{ParametersQuotient},  the hypotheses in (1) are satisfied and from the claim  we obtain the contradiction $o_1'^G=o_1'^H=o_1'^G+1$. Hence we have $o_2'^G=o_2'^H$. Therefore, if the claim fails, we necessarily have $o'^G=(0,o'_2)$ and $o'^H=(1,o'_2)$ and hence $p^{\max(n_1+1,n_2+o_2')}=\exp(H)=\exp(G)=p^{\max(n_1,n_2+o'_2)}$.  Therefore   $n_1+o_1'^G=n_1<n_2+o'_2=n_2+o_2'^G$ and \Cref{Main}\ref{4} yields $o_1=0$. It follows from $o^G\ne (0,0)$ that $o_2\ne 0$. This finishes the proof of the claim.

Combining the claim with \Cref{Main}\ref{4} we deduce that $n_i+o_i'^G\ge n_j+o_j'-o_j$ and hence, applying \ref{ExpCGG'o1=0} and \ref{ExpCGG'o2=0} of \Cref{G/L-N} we have 
\begin{equation}\label{ExpCGH}
\exp(\Cen_{\Gamma}(\Gamma'))=\begin{cases}
p^{n_i+o_i'^G}, & \text{if } \Gamma=G; \\
p^{n_i+o_i'^G+1}, & \text{if } \Gamma=H.
\end{cases}
\end{equation}
This yields a contradiction  to  \Cref{JenningsExpCGG'}. 
\end{proof}

This completes the proof \Cref{oo'Theorem}. 
As a consequence,  we   set  $o'^G=o'^H=(o_1',o_2')$  until the end of this section.  The following is the same as  \Cref{Cor1.0}.

\begin{corollary}\label{cor:Cent} One has
	 $\Cen_G(G')\cong \Cen_H(H')$.
\end{corollary}

\begin{proof} 
By   \cref{JenningsExpCGG'}\ref{JenningsCGG'}  we  have $|\M_i(\Cen_G(G'))/\M_{i+1}(\Cen_G(G'))|=|\M_i(\Cen_G(H'))/\M_{i+1}( \Cen_H(H'))| $ for each $i\geq 1$, hence also $|\M_i(\Cen_G(G'))|=|\M_i(\Cen_H(H'))|$ for each $i\geq 1$.    In order to prove the lemma we  will  write a presentation for a group $\Delta$ depending only on $ p,m,n_1,n_2,o_1,o_2,o_1',o_2' $ and  $e$, where $p^e=|\M_{p^{n_1-o_1+o_2}}(\Cen_\Gamma(\Gamma'))|$,  and show that $\Cen_\Gamma(\Gamma') \cong \Delta$ for $\Gamma\in\{G,H\}$.
We rely on the description of $C_\Gamma(\Gamma')$ given in \Cref{G/L-N}\ref{G/CGG'}  and analyze the different cases listed in \cref{Main}\ref{4}.
	
Assume first  that  $o_1=0$. Then $ \Cen_\Gamma(\Gamma')=	\GEN{ b_1,b_2^{p^{o_2}},a} $, and by \eqref{Commb1} and \cref{EseProp}\ref{EseCero}  we have $[b_2^{u_1^{\Gamma}p^{o_2}},b_1]=a^{u_1^{\Gamma}p^{o_2}}$. Defining 
  $$\Delta=\GEN{x,y,z\mid [y,x]=z^{p^{o_2}}, [z,x]=[z,y] =1, x^{{p^{n_1}}}=z^{p^{m-o_1'}}, y^{p^{n_2-o_2}}=z^{p^{m-o_2'}},z^{p^m}=1 },$$
the assignment $(x,y,z)\mapsto (b_1^{u_2^\Gamma},b_2^{   u_1^\Gamma p^{o_2}},a^{u_1^\Gamma u_2^\Gamma})$ extend to an isomorphism $\Delta \to \Cen_\Gamma(\Gamma')$.

Next assume that  $o_2=0<o_1$, so that 
		$ \Cen_\Gamma(\Gamma')=	\GEN{ b_1^{p^{o_1}},b_2,a}.$ Set 
	$$\Delta=\GEN{x,y,z\mid [y,x]=z^{p^{o_1}}, [z,x]=[z,y] =1, x^{{p^{n_1-o_1}}}=z^{p^{m-o_1'}}, y^{p^{n_2 }}=z^{p^{m-o_2'}},z^{p^m}=1 }.$$
	Then the assignment $(x,y,z)\mapsto (b_1^{u_2^\Gamma p^{o_1}}, b_2^{   u_1^\Gamma }, a^{u_1^\Gamma u_2^\Gamma})$ induces an isomorphism $\Delta \to \Cen_\Gamma(\Gamma')$.
 
 Finally assume that $o_1o_2>0$.  Then $o_2<o_1<o_2+n_1-n_2$ and $o'_1\le o'_2\le o'_1+n_1-n_2$,  and we also have 
 	$$\Cen_\Gamma(\Gamma')=	\GEN{ b_1^{p^{o_1}}, b_1^{p^{o_1-o_2}} b_2^{-1},a}=	\GEN{  b_1^{p^{o_1-o_2}} b_2^{-1},b_2^{p^{o_2}},a}.$$  	
 Let $r'$ be an integer such that $r_2r'\equiv 1 \mod p^m$. 
By the definition of $r_2$ in \eqref{Erres} we have $r'\equiv 1 \mod p^{m-o_2}$ and, as $o_1>o_2$,   we also have  $r'\equiv 1\mod p^{m-o_1}$. 
Applying \eqref{Commb2} we  get  
$[b_2,b_1^{p^{o_1-o_2}}b_2^{-1}]=a^{r'\Ese{r_1}{p^{o_1-o_2}}}$  and, as a consequence, that  $[b_2^{p^{o_2}},b_1^{p^{o_1-o_2}}b_2^{-1}]=
a^{r'\Ese{r_1}{p^{o_1-o_2}}{\Ese{r_2}{p^{o_2}}}}$. 
 From \Cref{EseProp}\ref{ValEse}-\ref{EseCero} and $r'\equiv 1 \mod p^{m-o_1}$   we derive
$r'\Ese{r_1}{p^{o_1-o_2}}{\Ese{r_2}{p^{o_2}}}\equiv r'p^{o_1} \equiv p^{o_1} \mod p^m$, from which it follows that $[b_2^{p^{o_2}},b_1^{p^{o_1-o_2}}b_2^{-1}]=a^{p^{o_1}}$.
As $ C_G(G') $ is of class at most $2$,  we have moreover that  
\begin{equation}\label{Commutator}
[b_2^{dp^{o_2}},(b_1^{p^{o_1-o_2}}b_2^{-1})^e]=a^{p^{o_1}de}, \quad \text{for every } d,e\in \Z.
\end{equation}
We   construct different $\Delta$'s depending on whether  $s=n_1+o'_1-n_2-o'_2-o_1+o_2$ is positive, negative or zero.  

 Suppose  $s>0$  and take $\alpha^\Gamma=u_1^\Gamma-u_2^\Gamma p^s$ and
	$$\Delta=\GEN{x,y,z\mid [y,x]=z^{p^{o_1}}, [z,x]=[z,y] =1, x^{{p^{n_1-o_1+o_2}}}=z^{p^{m-o_1'}}, y^{p^{n_2-o_2}}=z^{p^{m-o_2'}},z^{p^m}=1 }.$$ 
 Define additionally  
  We claim that the homomorphism  $f_\Gamma:\Delta\to \Cen_\Gamma(\Gamma' )$ that is defined by
  $$ x\mapsto x_1:=(b_1^{p^{o_1-o_2}}b_2^{-1})^{u_2^\Gamma} , \quad 
  y\mapsto y_1:= b_2^{\alpha^\Gamma p^{o_2}}  , \quad z\mapsto z_1:=a^{\alpha^\Gamma u_2^\Gamma}$$
 is in fact an isomorphism.
 Indeed, both $a$ and $z$ have order $p^m$ and $\Cen_\Gamma(\Gamma')/\GEN{a}\cong \Delta/\GEN{z}\cong C_{p^{n_1-o_1+o_2}}\times C_{p^{n_2-o_2}}$,  so    to prove that $f_{\Gamma}$ is an isomorphism  we  check that $x_1$, $y_1$ and $z_1$ satisfy the relations of $\Delta$. 
  It is clear that  $[z_1,x_1]=[z_1,y_1]=1$   and  regularity grants that  $x_1^{{p^{n_1-o_1+o_2}}}=z_1^{p^{m-o_1'}}$ and $y_1^{p^{n_2-o_2}}=z_1^{p^{m-o_2'}}$.
The  last  relation follows from \eqref{Commutator}. 

In case $s<0$ we take $\alpha^\Gamma=u_1^\Gamma p^s-  u_2^\Gamma $,
$$\Delta=\GEN{x,y,z\mid [y,x]=z^{p^{o_1}}, [z,x]=[z,y] =1, x^{{p^{n_1-o_1+o_2}}}=z^{p^{m-o_1'+s}}, y^{p^{n_2-o_2}}=z^{p^{m-o_2'}},z^{p^m}=1 },$$ 
and $f_\Gamma$ defined as in the previous paragraph. The same argument shows that $f_\Gamma$ is a group isomorphism. 

Finally suppose that $s=0$ and factorize $u_1^\Gamma-u_2^\Gamma=w^\Gamma p^{\ell^\Gamma}$ with $w^{\Gamma}$ coprime to $p$. Observe that $(w^G,\ell^G)$ and $(w^H,\ell^H)$ might be different,  nonetheless the following hold:
$$p^e = |\M_{p^{n_1-o_1+o_2}}(\Cen_\Gamma(\Gamma'))|=
| \mho_{n_1-o_1+o_2}( \Cen_\Gamma(\Gamma') )  |= 
\left|\GEN{a^{ {p^{\min\{m-o_1'+\ell^\Gamma, m-o_1'+o_2, n_1-o_1+o_2 \}}  }}}\right|.$$	
  As a result, $e=\max\{ o'_1-\ell^\Gamma,o'_1-o_2,m-n_1+o_1-o_2,0\}$ is independent of   $\Gamma\in \{G,H\}$.  We now give contructions of $\Delta$ depending on the value of $e$.

Assume that $e=\max\{o_1'-o_2,m-n_1+o_1-o_2,0\} $. We set
$$\Delta=\GEN{x,y,z\mid [y,x]=z^{p^{o_1}}, [z,x]=[z,y] =1, x^{{p^{n_1-o_1+o_2}}}=1, y^{p^{n_2-o_2}}=z^{p^{m-o_2'}},z^{p^m}=1 }$$
and select an integer $v$ such that $vu_2^{\Gamma}\equiv 1 \mod p^m$. 
Then we obtain an isomorphism $\Delta\rightarrow \Cen_\Gamma(\Gamma')$ by  assigning  $y\mapsto   b_2^{ u_1^\Gamma  p^{o_2}}$, $z\mapsto a^{u_1^\Gamma u_2^\Gamma}$ and
$$x\mapsto \begin{cases} 
	( b_1^{p^{o_1-o_2}}b_2^{-1})^{u_2^\Gamma}, & \text{if } e=0; \\
	( b_1^{p^{o_1-o_2}}b_2^{-1}  a^{-w^\Gamma p^{m-o_1'  -n_1+\ell^\Gamma +o_1-o_2}})^{u_2^\Gamma}, & \text{if } e=m-n_1+o_1-o_2; \\
	( b_1^{p^{o_1-o_2}}b_2^{-1+(u_2^\Gamma-u_1^\Gamma)v})^{u_2^\Gamma}, & \text{if } e=o'_1-o_2.
\end{cases}$$

Otherwise assume
 $e=o'_1-\ell^\Gamma>\max\{o_1'-o_2, m-n_1+o_1-o_2,0\}$ and we take
	$$\Delta= \GEN{x,y,z\mid [y,x]=z^{p^{o_1}}, [z,x]=[z,y] =1, x^{{p^{n_1-o_1+o_2}}}=z^{p^{m-e}}, y^{p^{n_2-o_2}}=z^{p^{m-o_2'}},z^{p^m}=1 }.$$
  Slightly modifying the arguments from the previous paragraph, one easily shows that 
$$ x\mapsto ( b_1^{p^{o_1-o_2}}b_2^{-1})^{u_2^\Gamma}  , \quad y\mapsto   b_2^{   w^\Gamma p^{ o_2 }}  , \qand z\mapsto a^{ w^\Gamma u_2^\Gamma  }$$
determines an isomorphism $\Delta\to \Cen_\Gamma(\Gamma')$.
\end{proof}

  The following result is the same as   \Cref{Cor1.1}.
	
	\begin{corollary}\label{cor:type}
The groups	$G$ and $H$ have the same type invariants.
	\end{corollary}

\begin{proof}
  We will  express the type invariants  solely  in terms of $(p,m,n_1,n_2,o_1,o_2,o_1',o_2')$.
The corollary then  will follow  from \Cref{oo'Theorem}.
  The discussion after \cite[Theorem 3.1]{Song2013} guarantees, for $\Gamma\in\{G,H\}$, that $\omega(\Gamma) \in \{2,3\}$ and, moreover, $\omega(\Gamma)=2$ if and only if $\Gamma$ is metacyclic, that is if $\max\{o_1',o_2'\}=m$.    In view of  \Cref{Known}\ref{it:3K}\ref{Metacyclic} we  assume  without loss of generality  that $\omega(\Gamma)=3$. In this situation, as  indicated in   \cite{Song2013}, the type invariants satisfy the following properties:
	\begin{equation}\label{PropertiesTypeInvariants}
	\exp(\Gamma)=p^{e_1^\Gamma}, \quad |\Gamma|=p^{e_1^\Gamma+e_2^\Gamma+e_3^\Gamma} \qand p^{e_3^\Gamma}= \min \{|g| : g\in \Phi(\Gamma)\setminus \mho_{1}(\Gamma)\}.
 	\end{equation}
	where  $\Phi(\Gamma)=\GEN{b_1^p,b_2^p,a}$ is the Frattini subgroup of $\Gamma$ and $e_i^\Gamma$ denotes the $i$-th type invariant of $\Gamma$.
	  We claim that  the type invariants   are given  by the following formulae:
	\begin{eqnarray} \label{ClaimTypeInvariants}
	\begin{split}
	e_1^\Gamma&= \max(n_1+o_1', n_2+o_2'), \\
	e_2^\Gamma&= n_1+n_2+ \max(o_1',o_2') - \max(n_1+o_1', n_2+o_2'), \\
	e_3^\Gamma&=m- \max(o_1',o_2').
	\end{split}
	\end{eqnarray} 
	Indeed, by \Cref{Characteristic}\ref{exponente} the exponent of $G$ is $p^{\max(n_1+o'_1,n_2+o'_2)}$, so \eqref{PropertiesTypeInvariants}  yields the first equality.
		Moreover the order of $G$ is both $p^{e_1^\Gamma+e_2^\Gamma+e_3^\Gamma}$ and $p^{m+n_1+n_2}$.
		Therefore it is enough to prove that $e_3^\Gamma=m-\max(o_1',o'_2)$. 
	  For this, let  $g$ in $\Phi(\Gamma)\setminus \mho_{1}(\Gamma)$. Then   $g=b_1^{xp^{s_1}} b_2^{y p^{s_2}} a^z$, with $1\leq \min\{s_1,s_2\}$ and $p$ does not divide $ zxy$.  Conversely, every element of such form belongs to $\Phi(\Gamma)\setminus \mho_{1}(\Gamma)$. Thanks to \eqref{Exp},  if $p^N\geq |g|$ then 
\begin{equation}\label{eq:invinv}
1=(b_1^{xp^{s_1}} b_2^{y p^{s_2}} a^z)^{p^N}= b_1^{x p^{s_1+N}} b_2^{y  p^{s_2+N}} a^{\Ese{r_1}{xp^{s_1}} \Ese{r_2}{y p^{s_2}}  \Te{r_1^{x p^{s_1}}, r_2^{y p^{s_2}}}{p^N}  +z\Ese{r_1^{xp^{s_1}} r_2^{y p^{s_2}}   }{p^N}}.
\end{equation}		
		Then $s_i+N\geq n_i $ for $i=1,2$  and, as $n_1\geq m$, \Cref{EseProp}\ref{ValEse}-\ref{ValEse2}  yields 
		$$v_p\left(\Ese{r_1}{xp^{s_1}} \Te{r_1^{x p^{s_1}}, r_2^{y p^{s_2}}}{p^N}\right)\ge  s_1+N\geq n_1\geq  m  \qand
		v_p\left(z\Ese{r_1^{xp^{s_1}} r_2^{y p^{s_2}}   }{p^N}\right)=N.$$
		So $\Ese{r_1^{xp^{s_1}} r_2^{y p^{s_2}}   }{p^N}= A  p^N$, for some integer $A=A(x,y,s_1,s_2)$, with $p$ not dividing $A$.
		Hence   \eqref{eq:invinv} can be rewritten as 
		$$1= a^{x u_1^\Gamma p^{s_1+N -n_1 + m-o_1'}  + y u_2 ^\Gamma p ^{s_2+N -n_2+m-o_2'}  +    zA p^N}$$
		and $e_3^\Gamma$ is  the minimum value of $N$ such that there are integers $x,y, z,s_1$ and $s_2$ satisfying 
		\begin{equation}\label{Congruence1}
		\begin{split}
		p\nmid xyz, \quad 1\le s_i \quad  s_i+N\ge n_i, \qand & \\ 
		x u_1^\Gamma p^{s_1+N -n_1 + m-o_1'}  + y u_2^\Gamma p ^{s_2+N -n_2+m-o_2'} + zA p^N & \equiv 0 \mod p^m.
		\end{split}
		\end{equation}  
		If $N\geq m$ then the conditions hold trivially taking $s_1$ and $s_2$ large enough, so we look for values $N<m$. 
		Then $zA p^N\not\equiv 0 \mod p^m$ and hence $s_1$ and $s_2$ should be taken satisfying one of the following conditions:
		$$\text{(1) } s_1=n_1+o_1'-m, \quad 
		\text{(2) } s_2=n_2+o_2'-m, \quad 
		\text{(3) }   m<n_1+o_1'-s_1=n_2+o_2'-s_2  .$$
		In the three cases $n_i\le s_i+N\le n_i + o_i'-m+N$ and hence $N\ge m-o'_i$. 
		This shows that $$N\ge \min(m-o'_1,m-o'_2)=m-\max(o'_1,o'_2)$$ and hence,   in view of  the above we   assume that $\max(o'_1,o'_2)>0$. It now suffices to prove that, for $N=m-\max(o'_1,o'_2)$, there exists integers $x,y,z,s_1$ and $s_2$ satisfying \eqref{Congruence1}. From $N<m$ we conclude that $\max(o'_1,o'_2)>0$. 
  We note, moreover  that $n_i+o'_i\ge m$. Indeed,    if this weren't the case, we would have $n_i<m$ and, applying \Cref{Main}\ref{2}-\ref{5}, that   $i=2$   and $n_2=2m-o_1-o'_2> m-o'_2$. 
		If $o'_1\ge o'_2$ then the following integers satisfy \eqref{Congruence1}: 
$$N=m-o_1',\quad s_1=n_1+o'_1-m=n_i-N>0, \quad s_2=n_2+o'_2\ge m,\quad x=-1,\quad y=1,\quad z=Bu_1$$		
		with $B$ an integer such that $AB   \equiv 1 \mod p^m$.
		Otherwise $o'_2>o'_1$ and a symmetric argument shows that for $N=m-o_2'$ the integers  
		$$x=1, \quad y=-1, \quad s_1=n_1+o_1'\ge m, \quad s_2=n_2-m+o_2'=n_2-N>0, \quad z=Bu_2$$ 
		satisfy \eqref{Congruence1}. This finishes the proof of \eqref{ClaimTypeInvariants} and hence also of the corollary.
\end{proof}

 \bibliographystyle{amsalpha}
 \bibliography{MIP}

\providecommand{\bysame}{\leavevmode\hbox to3em{\hrulefill}\thinspace}
\providecommand{\MR}{\relax\ifhmode\unskip\space\fi MR }
\providecommand{\MRhref}[2]{%
  \href{http://www.ams.org/mathscinet-getitem?mr=#1}{#2}
}
\providecommand{\href}[2]{#2}
\begin{thebibliography}{GLMdR22}

\bibitem[Bag88]{BaginskiMetacyclic}
C.~Bagi\'{n}ski, \emph{The isomorphism question for modular group algebras of
  metacyclic {$p$}-groups}, Proc. Amer. Math. Soc. \textbf{104} (1988), no.~1,
  39--42.

\bibitem[Bag99]{Bag99}
\bysame, \emph{On the isomorphism problem for modular group algebras of
  elementary abelian-by-cyclic {$p$}-groups}, Colloq. Math. \textbf{82} (1999),
  no.~1, 125--136.

\bibitem[BC88]{BC88}
C.~Bagi\'{n}ski and A.~Caranti, \emph{The modular group algebras of
  {$p$}-groups of maximal class}, Canad. J. Math. \textbf{40} (1988), no.~6,
  1422--1435.

\bibitem[BdR21]{BdR20}
O.~Broche and \'{A}. del R\'{\i}o, \emph{The {M}odular {I}somorphism {P}roblem
  for two generated groups of class two}, Indian J. Pure Appl. Math.
  \textbf{52} (2021), 721--728.

\bibitem[BGLdR21]{OsnelDiegoAngel}
O.~Broche, D.~Garc\'{\i}a-Lucas, and \'{A}. del R\'{\i}o, \emph{A
  classification of the finite two-generated cyclic-by-abelian groups of prime
  power order}, http://arxiv.org/abs/2106.06449.

\bibitem[BK07]{BK07}
C.~Bagi\'{n}ski and A.~Konovalov, \emph{The modular isomorphism problem for
  finite {$p$}-groups with a cyclic subgroup of index {$p^2$}}, Groups {S}t.
  {A}ndrews 2005. {V}ol. 1, London Math. Soc. Lecture Note Ser., vol. 339,
  Cambridge Univ. Press, Cambridge, 2007, pp.~186--193.

\bibitem[BKRW99]{BKRW99}
F.~M. Bleher, W.~Kimmerle, K.~W. Roggenkamp, and M.~Wursthorn,
  \emph{Computational aspects of the isomorphism problem}, Algorithmic algebra
  and number theory ({H}eidelberg, 1997), Springer, Berlin, 1999, pp.~313--329.

\bibitem[Bra63]{Bra63}
R.~Brauer, \emph{Representations of finite groups}, Lectures on {M}odern
  {M}athematics, {V}ol. {I}, Wiley, New York, 1963, pp.~133--175.

\bibitem[Che82]{Cheng1982}
Y.~Cheng, \emph{On finite p-groups with cyclic commutator subgroup}, Archiv der
  Mathematik \textbf{39} (1982), no.~4, 295--298.

\bibitem[CR81]{CurtisReiner1981}
Ch.~W. Curtis and I.~Reiner, \emph{Methods of representation theory. {V}ol.
  {I}}, John Wiley \& Sons Inc., New York, 1981, With applications to finite
  groups and orders, Pure and Applied Mathematics, A Wiley-Interscience
  Publication. \MR{632548 (82i:20001)}

\bibitem[CR87]{CurtisReiner1987}
\bysame, \emph{Methods of representation theory. {V}ol. {II}}, Pure and Applied
  Mathematics (New York), John Wiley \& Sons Inc., New York, 1987, With
  applications to finite groups and orders, A Wiley-Interscience Publication.
  \MR{892316 (88f:20002)}

\bibitem[Dad71]{Dade71}
E.~Dade, \emph{Deux groupes finis distincts ayant la m\^{e}me alg\`ebre de
  groupe sur tout corps}, Math. Z. \textbf{119} (1971), 345--348.

\bibitem[Des56]{Deskins1956}
W.~E. Deskins, \emph{Finite {A}belian groups with isomorphic group algebras},
  Duke Math. J. \textbf{23} (1956), 35--40. \MR{77535}

\bibitem[Eic08]{Eick08}
B.~Eick, \emph{Computing automorphism groups and testing isomorphisms for
  modular group algebras}, J. Algebra \textbf{320} (2008), no.~11, 3895--3910.

\bibitem[EK11]{EK11}
B.~Eick and A.~Konovalov, \emph{The modular isomorphism problem for the groups
  of order 512}, Groups {S}t {A}ndrews 2009 in {B}ath. {V}olume 2, London Math.
  Soc. Lecture Note Ser., vol. 388, Cambridge Univ. Press, Cambridge, 2011,
  pp.~375--383.

\bibitem[GLMdR22]{GarciaMargolisdelRio}
D.~García-Lucas, L.~Margolis, and Á. del Río, \emph{Non-isomorphic
  $2$-groups with isomorphic modular group algebras}, J. Reine Angew. Math.
  \textbf{154} (2022), no.~783, 269--274.

\bibitem[Her01]{Hertweck2001}
M.~Hertweck, \emph{A counterexample to the isomorphism problem for integral
  group rings}, Ann. of Math. (2) \textbf{154} (2001), no.~1, 115--138.

\bibitem[Hig40]{Hig40}
G.~Higman, \emph{The units of group-rings}, Proc. London Math. Soc. (2)
  \textbf{46} (1940), 231--248.

\bibitem[HS06]{HS06}
M.~Hertweck and M.~Soriano, \emph{On the modular isomorphism problem: groups of
  order {$2^6$}}, Groups, rings and algebras, Contemp. Math., vol. 420, Amer.
  Math. Soc., Providence, RI, 2006, pp.~177--213.

\bibitem[HS07]{HertweckSoriano07}
Martin Hertweck and Marcos Soriano, \emph{Parametrization of central {F}rattini
  extensions and isomorphisms of small group rings}, Israel J. Math.
  \textbf{157} (2007), 63--102.

\bibitem[Hup67]{Hup67}
B.~Huppert, \emph{Endliche {G}ruppen. {I}}, Die Grundlehren der Mathematischen
  Wissenschaften, Band 134, Springer-Verlag, Berlin-New York, 1967.

\bibitem[Jen41]{Jen41}
S.~A. Jennings, \emph{The structure of the group ring of a {$p$}-group over a
  modular field}, Trans. Amer. Math. Soc. \textbf{50} (1941), 175--185.

\bibitem[Kim91]{KimmerleHabil}
W.~Kimmerle, \emph{Beitr\"age zur ganzzahligen {D}arstellungstheorie endlicher
  {G}ruppen}, Bayreuth. Math. Schr. (1991), no.~36, 139.

\bibitem[Kü82]{Kulshammer1982}
B.~Külshammer, \emph{Bemerkungen \"{u}ber die {G}ruppenalgebra als
  symmetrische {A}lgebra. {II}}, J. Algebra \textbf{75} (1982), no.~1, 59--69.

\bibitem[Mak76]{Makasikis}
A.~Makasikis, \emph{Sur l'isomorphie d'alg\`ebres de groupes sur un champ
  modulaire}, Bull. Soc. Math. Belg. \textbf{28} (1976), no.~2, 91--109.
  \MR{561324}

\bibitem[{Mar}22]{Mar22}
L.~{Margolis}, \emph{{The Modular Isomorphism Problem: A Survey}}, Jahresber.
  Dtsch. Math. Ver. (2022).

\bibitem[MM20]{MM20}
L.~{Margolis} and T.~{Moede}, \emph{{The Modular Isomorphism Problem for small
  groups -- revisiting Eick's algorithm}}, arXiv:2010.07030,
  https://arxiv.org/abs/2010.07030.

\bibitem[MS22]{MS22}
L.~Margolis and M.~Stanojkovski, \emph{On the modular isomorphism problem for
  groups of class $3$ and obelisks}, J. Group Theory \textbf{25} (2022), no.~1,
  163--206.

\bibitem[MST21]{MSS21}
L.~Margolis, M.~Stanojkovski, and Sakurai T., \emph{Abelian invariants and a
  reduction theorem for the modular isomorphism problem},
  https://arxiv.org/abs/2110.10025.

\bibitem[NS18]{NavarroSambale}
G.~Navarro and B.~Sambale, \emph{On the blockwise modular isomorphism problem},
  Manuscripta Math. \textbf{157} (2018), no.~1-2, 263--278.

\bibitem[Pas65]{Passman65}
D.~S. Passman, \emph{Isomorphic groups and group rings}, Pacific J. Math.
  \textbf{15} (1965), 561--583.

\bibitem[Pas77]{Pas77}
\bysame, \emph{The algebraic structure of group rings}, Pure and Applied
  Mathematics, Wiley-Interscience [John Wiley \& Sons], New York-London-Sydney,
  1977.

\bibitem[PW50]{PW50}
S.~Perlis and G.~L. Walker, \emph{Abelian group algebras of finite order},
  Trans. Amer. Math. Soc. \textbf{68} (1950), 420--426.

\bibitem[Qui68]{Qui68}
Daniel~G. Quillen, \emph{On the associated graded ring of a group ring}, J.
  Algebra \textbf{10} (1968), 411--418.

\bibitem[RS87]{RoggenkampScott1987}
K.~W. Roggenkamp and L.~Scott, \emph{Isomorphisms of {$p$}-adic group rings},
  Ann. of Math. (2) \textbf{126} (1987), no.~3, 593--647.

\bibitem[Sak20]{Sakurai}
T.~Sakurai, \emph{The isomorphism problem for group algebras: a criterion}, J.
  Group Theory \textbf{23} (2020), no.~3, 435--445.

\bibitem[Sal93]{SalimTesis}
M.~A.~M. Salim, \emph{The isomorphism problem for the modular group algebras of
  groups of order $p^5$}, ProQuest LLC, Ann Arbor, MI, 1993, Thesis
  (Ph.D.)--University of Manchester.

\bibitem[San84]{San84Ab}
R.~Sandling, \emph{Units in the modular group algebra of a finite abelian
  {$p$}-group}, J. Pure Appl. Algebra \textbf{33} (1984), no.~3, 337--346.

\bibitem[San85]{Sandling85}
\bysame, \emph{The isomorphism problem for group rings: a survey}, Orders and
  their applications ({O}berwolfach, 1984), Lecture Notes in Math., vol. 1142,
  Springer, Berlin, 1985, pp.~256--288.

\bibitem[San89]{San89}
\bysame, \emph{The modular group algebra of a central-elementary-by-abelian
  {$p$}-group}, Arch. Math. (Basel) \textbf{52} (1989), no.~1, 22--27.

\bibitem[San96]{San96}
\bysame, \emph{The modular group algebra problem for metacyclic {$p$}-groups},
  Proc. Amer. Math. Soc. \textbf{124} (1996), no.~5, 1347--1350.

\bibitem[Seh67]{Sehgal1967}
S.~K. Sehgal, \emph{On the isomorphism of group algebras}, Math. Z. \textbf{95}
  (1967), 71--75.

\bibitem[Seh78]{Seh78}
\bysame, \emph{Topics in group rings}, Monographs and Textbooks in Pure and
  Applied Math., vol.~50, Marcel Dekker, Inc., New York, 1978.

\bibitem[Son13]{Song2013}
Q.~Song, \emph{Finite two-generator {$p$}-subgrous with cyclic derived group},
  Comm. Algebra \textbf{41} (2013), no.~4, 1499--1513.

\bibitem[SS95]{SalimSandling1995}
M.~A.~M. Salim and R.~Sandling, \emph{The unit group of the modular small group
  algebra}, Math. J. Okayama Univ. \textbf{37} (1995), 15--25 (1996).

\bibitem[SS96]{SalimSandlingp5}
\bysame, \emph{The modular group algebra problem for groups of order {$p^5$}},
  J. Austral. Math. Soc. Ser. A \textbf{61} (1996), no.~2, 229--237.

\bibitem[War61]{Ward}
H.~N. Ward, \emph{Some results on the group algebra of a $p$-group over a prime
  field}, Seminar on finite groups and related topics., Mimeographed notes,
  Harvard Univ., 1960-61, pp.~13--19.

\bibitem[Whi68]{Whitcomb}
A.~Whitcomb, \emph{The {G}roup {R}ing {P}roblem}, ProQuest LLC, Ann Arbor, MI,
  1968, Thesis (Ph.D.)--The University of Chicago.

\bibitem[Wur93]{Wursthorn1993}
M.~Wursthorn, \emph{Isomorphisms of modular group algebras: an algorithm and
  its application to groups of order {$2^6$}}, J. Symbolic Comput. \textbf{15}
  (1993), no.~2, 211--227. \MR{1218760}

\end{thebibliography}
 
 	\bigskip \medskip
	
	\noindent
	\footnotesize 
	{\bf Authors' addresses:}
	
	\smallskip
		
	\noindent Diego Garc\'ia-Lucas, Departamento de Matem\'aticas, Universidad de Murcia
	\hfill {\tt  diego.garcial@um.es}
		
	\noindent \'Angel del R\'io, Departamento de Matem\'aticas, Universidad de Murcia
	\hfill {\tt adelrio@um.es}
	
	\noindent Mima Stanojkovski,  RWTH Aachen and
	MPI-MiS Leipzig
	\hfill {\tt mima.stanojkovski@rwth-aachen.de}

\end{document}